\documentclass[11pt]{amsart}
\usepackage[margin=0.9in]{geometry}
\usepackage{amsmath}
\usepackage{amsfonts}
\usepackage{amssymb}
\usepackage{graphicx}
\usepackage{mathrsfs}
\usepackage{graphicx,color}
\usepackage[mathscr]{eucal}
\usepackage{comment}
\usepackage[all]{xy}
\usepackage{hyperref}
\hypersetup{colorlinks=true, linkcolor=black, citecolor=black}
\usepackage{tikz}
\usepackage{tikz-cd}
\usepackage{graphicx}
\usepackage[shortlabels]{enumitem}
\usetikzlibrary{decorations.pathreplacing}

\newtheorem{theorem}{Theorem}[section]
\newtheorem{lemma}[theorem]{Lemma}
\newtheorem{proposition}[theorem]{Proposition}
\newtheorem{corollary}[theorem]{Corollary}

\theoremstyle{definition}
\newtheorem{definition}[theorem]{Definition}
\newtheorem{example}[theorem]{Example}

\theoremstyle{remark}
\newtheorem{remark}[theorem]{Remark}

\numberwithin{equation}{section}

\newcommand{\abs}[1]{\lvert#1\rvert}

\newcommand{\NN}{\mathbb{N}}
\newcommand{\ZZ}{\mathbb{Z}}
\newcommand{\QQ}{\mathbb{Q}}
\newcommand{\RR}{\mathbb{R}}
\newcommand{\CC}{\mathbb{C}}

\newcommand{\KK}{\mathbb{K}}

\title{Contact big fiber theorems}

\date{\today}

\author{Yuhan Sun}
\address{Huxley Building,
South Kensington Campus,
Imperial College London,
London, SW7 2AZ, U.K}
\email{yuhan.sun@imperial.ac.uk}

\author{Igor Uljarevi\'c}
\address{Faculty of Mathematics, University of Belgrade, Studentski trg 16, Belgrade, Serbia}
\email{igor.uljarevic@matf.bg.ac.rs}

\author{Umut Varolgunes}
\address{Mathematics Department, Koç University Rumelifeneri Campus, Istanbul,Turkey}
\email{uvarolgunes@ku.edu.tr}

\begin{document}

\begin{abstract}
  We prove contact big fiber theorems, analogous to the symplectic big fiber theorem by Entov and Polterovich,  using symplectic cohomology with support. Unlike in the symplectic case, the validity of the statements requires conditions on the closed contact manifold. One such condition is to admit a Liouville filling with non-zero symplectic cohomology. In the case of Boothby-Wang contact manifolds, we prove the result under the condition that the Euler class of the circle bundle, which is the negative of an integral lift of the symplectic class, is not an invertible element in the quantum cohomology of the base symplectic manifold. As applications, we obtain new examples of rigidity of intersections in contact manifolds and also of contact non-squeezing. 
\end{abstract}

\maketitle

\tableofcontents

\section{Introduction}

Consider a closed symplectic manifold $M$ and a smooth map $\pi=(\pi_1,\ldots,\pi_N): M\to \mathbb{R}^N$ whose components $\pi_1,\ldots,\pi_N$ pairwise Poisson commute. The celebrated symplectic big fiber theorem of Entov-Polterovich says that $\pi$ admits at least one fiber that is not displaceable from itself by a Hamiltonian diffeomorphism \cite{EP06}. In this paper, we prove contact versions of this theorem.

Let $(C,\xi)$ be a closed contact manifold and consider a smooth map $\pi: C \to \mathbb{R}^N$. Recall that once a contact form is chosen, we can define the Jacobi bracket\footnote{Recall that the Jacobi bracket $[f,g]$ of two smooth functions $f,g:C\to\mathbb{R}$ is defined as $[f,g]:= \alpha([X_f, X_g])$, where $\alpha$ is the chosen contact form and $X_f$ stands for the contact Hamiltonian vector field of $f$ with respect to $\alpha$.} on smooth functions of $C$. We assume that for some contact form $\alpha$ on $C$ the components of $\pi$ are invariant under the Reeb flow and they pairwise Jacobi commute (cf. contact integrable systems \cite{KT, JJ}). Let us call such a $\pi$ a contact involutive map witnessed by $\alpha$. We say that a subset of a contact manifold is contact displaceable if it is  displaceable from itself by a contact isotopy throughout the paper. The standard contact sphere $(S^{2n-1},\xi_{st})$, for $n\geqslant 2$, shows that without any condition there can be contact involutive maps whose fibers are all contact displaceable. Indeed, any non-constant smooth Reeb invariant function\footnote{Such a function can be constructed by lifting a smooth function from $\mathbb{C}P^{n-1}$, via the Hopf fibration.} on $S^{2n-1}$ can be seen as a contact involutive map whose all fibers are compact and different from the whole $S^{2n-1}$. Since the standard $S^{2n-1}$ without a point is contactomorphic to the standard $\mathbb{R}^{2n-1}$ and since any compact subset of the standard $\mathbb{R}^{2n-1}$ can be displaced by a compactly supported contact isotopy, this implies that there exist contact involutive maps on the standard $S^{2n-1}$, for $n\geqslant 2$ violating the conclusion of a potential contact big fiber theorem. 

Among our results, the easiest one to state is that if $C$ admits a Liouville filling with non-vanishing symplectic cohomology, then at least one fiber of $\pi$ is contact non-displaceable. In fact, we show that there must be a self-intersection point of this non-displaceable fiber that has $\alpha$-\emph{conformal factor} equal to $1.$ We call this the contact big fiber theorem.

We prove a contact big fiber theorem when $(C,\alpha)$ is a prequantization circle bundle over a symplectic manifold $D$ with an integral lift $\sigma$ of the symplectic class. We assume the technical conditions that are required for the Floer Gysin sequence of \cite{AK, BKK} to hold true and moreover, crucially, that quantum multiplication with $\sigma$ is not an invertible operator on $H^*(D; \Bbbk)[T, T^{-1}]$ for some commutative ring $\Bbbk.$ Under this condition, we also obtain results on non-displacability of Legendrians from their Reeb closures.

\begin{remark}  During the course of writing our result, we learned that an analogous result was conjectured by Albers-Shelukhin-Zapolsky around 2017.
    A closely related published work is by Borman and Zapolsky \cite{BZ}. They use Givental's nonlinear Maslov index \cite{Giv} to prove contact big fiber theorems \cite[Theorem 1.17]{BZ} for certain prequantization bundles over toric manifolds. Later Granja-Karshon-Pabiniak-Sandon \cite{GKPS} give an analogue for lens spaces of Givental’s construction. Our results here do not require the contact manifold to be a toric prequantization bundle.
\end{remark}

Our proof utilizes what one might call descent for Rabinowitz Floer homology. Let $\pi: C\to \RR^N$ be a contact involutive map witnessed by a contact form $\alpha.$ For simplicity of discussion at this point, consider a Liouville filling $\bar{M}$ of $(C,\alpha)$. Let ${M}$ be the result of attaching the semi-infinite symplectization of $(C,\alpha)$ to $\bar{M}.$ From our viewpoint, Rabinowitz Floer homology is nothing but the symplectic cohomology with support\footnote{Symplectic cohomology with support, also known as \emph{relative symplectic cohomology}, is discussed in Section~\ref{sec:SHwSupp}, and in more details in Section~\ref{sec:def-sh}.} on $\partial\bar{M}$ inside ${M}$, which enjoys a local-to-global principle. A finite cover of $C$ by preimages of compact subsets from the base of $\pi$ gives a Poisson commuting cover of $\partial\bar{M}$ inside ${M}$. A contact displaceable set in $C$ is Hamiltonian displaceable in $M$. Finally, we use the standard argument that displaceability implies the vanishing of invariant and the Mayer-Vietoris property to conclude that Rabinowitz Floer homology, and therefore, symplectic cohomology of $M$ vanishes. 
\newline

Now we go into a more detailed discussion of our general tools and results. Various concrete applications will be shown in Section \ref{sec:app}.

\subsection{Symplectic cohomology with support}\label{sec:SHwSupp} %of Liouville manifolds%
% YS: I deleted the words just for display reasons.

Let us fix a symplectic manifold $({M}^{2n},\omega)$ which is convex at infinity, which means that it admits a symplectic embedding of $([1,+\infty)\times C, d(r\alpha))$ with pre-compact complement, where $C^{2n-1}$ is a closed manifold with contact form $\alpha$ and $r$ denotes the coordinate on $[1,+\infty).$

Let $\Bbbk$ be a commutative ring. For any compact subset $K$ of $M$, we can define the symplectic cohomology of $M$ with support on $K$, denoted by $SH_M^*(K)$ \cite{Var21,GV,BSV}. See Section \ref{sec:def-sh} for the precise definition. We note that we will only consider contractible $1$-periodic orbits in the construction. Let us remind the reader of some basic properties.

\begin{enumerate}

      \item If $K$ is displaceable from itself in $M$ by a Hamiltonian diffeomorphism, then $SH_M^*(K)=0$, see \cite[Section 4.2]{Var} or \cite[Theorem 3.6]{BSV}.
        \item Assume that $K$ is a compact Liouville subdomain of a complete Liouville manifold $M$. Then $SH^*_M(K)$ is the Viterbo symplectic cohomology of the symplectic completion of $K$ \cite{Se,Vit}. Restriction maps agree with the Viterbo transfer maps. See Appendix \ref{sec:comparison-1}.
        \item Assume that $\Bbbk$ is a field and $K$ is a compact Liouville subdomain of a complete Liouville manifold $M$. Then $SH^*_M(\partial K)$ is isomorphic to the Rabinowitz Floer homology of $\partial K\subset K$ as defined by Cieliebak-Frauenfelder-Oancea \cite{CFO,CO}. See Appendix \ref{sec:comparison-2}.
\end{enumerate}

\begin{comment}
    Given a contact form on the ideal contact boundary $(C,\xi),$ we obtain an associated hypersurface $\partial\bar{M}\subset M$ \textcolor{blue}{[We should say somewhere what $\bar{M}$ is.], YS: added a definition.} which is canonically identified with $C$. Here $\bar{M}$ is a Liouville domain defined in Definition \ref{d:Liouville}. Therefore, a contact form on $(C,\xi)$ allows us to think of compact subsets of $C$ as compact subsets of $M$ disjoint from the skeleton. 
\end{comment}

%YS: it seems this paragraph has been already explained above.

We now explain the most useful property of symplectic cohomology with support for this paper. It is a combination of the Mayer-Vietoris principle \cite{Var21} and the displacement property.

\begin{definition}\label{d:weakly1}
    A collection of compact subsets $K_1,\cdots, K_n$ of $M$ is called Poisson commuting if there exist open neighborhoods $U_{m,i}\supset K_m$ and smooth functions $f_{m,i}:U_{m,i}\to \RR$ for $m\in\{1,\cdots,n\}$ and $i\in\ZZ_{+}$ such that
    \begin{enumerate}
        \item $K_m=\bigcap_{i=1}^\infty \{f_{m,i}< 0\},$ for all $m.$
        \item $\{f_{m,i}\leq \epsilon\}\subset U_{m,i}$ is compact for some $\epsilon >0$.
        \item $f_{m,i}<f_{m,i+1}$ on $U_{m,i}\cap U_{m,i+1}$ for all ${m,i}$.
        \item The Poisson bracket $\{f_{m,i},f_{m',i}\}=0$ on $U_{m,i}\cap U_{m',i}$ for all $m,m',i$.
    \end{enumerate}
\end{definition}

\begin{theorem}\label{t:main rel}
    Let ${K}_1, \cdots,{K}_N$ be a Poisson commuting collection of compact subsets of $M$. Let $K:=K_1\cup\ldots\cup K_N.$ If $SH^*_M(K)\neq 0,$ then at least one of ${K}_1, \cdots,{K}_N$ is not displaceable from itself by a Hamiltonian diffeomorphism.
\end{theorem}

We now fix a codimension zero submanifold $\bar{M}\subset M$ with convex boundary, which gives rise to a decomposition 
$$
M:= \bar{M}\cup_{\partial\bar{M}} ([1,+\infty)\times \partial\bar{M}).
$$
Motivated by the \textit{selective symplectic homology} \cite{U}, it is possible to define symplectic cohomology with support on non-compact subsets which are invariant under the Liouville flow at infinity. In this article we only use two such sets of the simplest form: $M$ itself and $K_M:=[1,+\infty)\times \partial\bar{M}$. The following Mayer-Vietoris sequence can be proved using the strategy of \cite{Var21}.

\begin{theorem}\label{thm-rfh-sh-vanish}
There is an exact sequence of $\Bbbk$-modules
$$
\cdots \to SH^k_M(M) \to SH^k_M(\bar{M})\oplus SH^k_M(K_M)\to SH^k_M(\partial\bar{M})\to \cdots,
$$ 
where $SH^*_M(M)$ is isomorphic to $H^*_c(M;\Lambda)$, the cohomology with compact support of $M$. 

Moreover, if we equip $H^*_c(M;\Lambda)$ with the quantum product, the map $H^*_c(M;\Lambda)\cong SH^*_M(M)\to SH^*_M(\bar{M})\cong SH^*(M)$ is a map of algebras.
\end{theorem}

The idea of the proof of the following statement is the same as that of \cite[Theorem 13.3]{Ritter}.

\begin{corollary}\label{c:boundary vanishing}
Assume that $H^*_c(M;\Lambda)$, equipped with the quantum product, is a nilpotent algebra, then $SH_M^*(\bar{M})=0$ if and only if $SH^*_M(\partial \bar{M})=0$.
\end{corollary}
\begin{proof}
    By the unitality of restriction maps \cite{TVar}\footnote{In general, $SH^\ast_M(K)$ has a unit if $K$ is compact (the construction of the unit element and the proof that the restriction maps preserve it follows the strategy of \cite{TVar}). In the cases that are relevant to us in this paper, namely when $K=\bar{M}$ and $K=\partial \bar{M}$, the symplectic cohomology with support $SH_M^\ast(K)$ coincides with the Viterbo symplectic cohomology $SH^\ast(M)$ and Rabinowitz Floer cohomology $RFH^\ast(\partial \bar{M})$, respectively (cf. Proposition~\ref{p:rel2vit} and Corollary \ref{c:RFH}). The unitality statements are well-known in this special case (see the proof of \cite[Theorem 13.3]{Ritter}). As the case $K=M$ shows, when $K$ is not compact, $SH^\ast_M(K)$ may not be unital.}, vanishing of $SH^*_M(\bar{M})$ implies the vanishing of $SH^*_M(\partial\bar{M})$. Conversely, if $SH^*_M(\partial\bar{M})=0$, then, by Theorem \ref{thm-rfh-sh-vanish}, $SH^*_M(M) \to SH^*_M(\bar{M})$ is surjective. This means, by the second part of Theorem \ref{thm-rfh-sh-vanish}, that the unit of $SH^*_M(\bar{M})$ is nilpotent and therefore $SH_M^*(\bar{M})=0$.
\end{proof}

If either $c_1(TM)$ or $\omega$ is zero on $\pi_2(M)$, then $H^*_c(M;\Lambda)$ is nilpotent. Combined with the comparison results in Appendix \ref{sec:comparison}, we get an alternative proof of the fact that for Liouville manifolds the vanishing of Rabinowitz Floer cohomology is equivalent to the vanishing of symplectic cohomology \cite{Ritter}, over field coefficients. Our argument works over the integers as well.

\subsection{Symplectic and contact big fiber theorems}
Now we give applications of our invariants to symplectic and contact topology. The famous \emph{symplectic big fiber theorem} by Entov-Polterovich \cite{EP06,EP09,DGPZ} says that any involutive map on a closed symplectic manifold has a non-displaceable fiber. We prove a symplectic big fiber theorem for Liouville manifolds as the most basic application of our methods.

\begin{corollary}\label{c:Liouville big fiber}
    Let $(M,\lambda)$ be a finite type complete Liouville manifold with $SH^*(M)\neq 0$. Any proper involutive map $F: M\to \RR^N$ has a fiber which is not Hamiltonian displaceable.
\end{corollary}
\begin{proof}
    Recall that $(M,\lambda)$ is the symplectic completion of some Liouville domain $\bar{M}$. Suppose for any $b\in F(\bar{M})$ the compact set $F^{-1}(b)$ is displaceable. Then any $b\in F(\bar{M})$ has a neighborhood $U_b$ whose closure $\overline{U}_b$ is compact such that $F^{-1}(\overline{U}_b)$ is displaceble in $M$. The cover $\{U_b\}$ admits a finite subcover $\{U_i\}$ since $F(\bar{M})$ is compact. By pulling back, the finite collection $\{\overline{U}_i\}$ gives a Poisson commuting cover $\{K_i=F^{-1}(\overline{U}_i)\}$ of $\bar{M}$. Hence Theorem \ref{t:main rel} shows that $SH^*_M(\cup K_i)=0$. By the unitality of the restriction map $SH^*_M(\cup K_i)\to SH^*_M(\bar{M})$, we get $SH^*_M(\bar{M})=0$, which is isomorphic to $SH^*(M)$, a contradiction.
\end{proof}

The non-vanishing condition for symplectic cohomology is necessary. The map $\abs{z}^2: \CC\to \RR$ has no rigid fiber. Further rigidity of intersections in Liouville manifolds will be described in Section \ref{sec:quadrics}. Particularly, we recover a result of Abouzaid-Diogo \cite{AD} on the cotangent bundle of a $3$-sphere.\newline

Next we introduce a contact big fiber theorem. Let us recast the definition of a contact involutive map in more symplectic terms. The equivalence is elementary and well-known.

\begin{definition}\label{d:contact Poisson}
    Let $(C,\xi)$ be a closed contact manifold and let $\pi: S(C)\to C$ be the projection from its symplectization. A smooth map $G=(g_1,\cdots,g_N):C\to\RR^N$ is called contact involutive if the function
    $$
    ( g_1\circ\pi,\cdots, g_N\circ\pi, r_\alpha): S(C)\to \RR^{N+1}
    $$
    is an involutive map on $S(C)$, for some contact form $\alpha$, where $r_\alpha$ denotes the Liouville coordinate.
\end{definition}

\begin{definition}
    Let $(C,\xi)$ be a contact manifold with a contact form $\alpha.$ Let $\phi: C\to C$ be a contactomorphism. We say that $p\in C$ has conformal factor $1$ with respect to $\phi$ and $\alpha$ if $(\phi^*\alpha)_p=\alpha_p.$
\end{definition}

The next theorem uses the notion of \emph{strong symplectic filling}. Recall that a compact symplectic manifold $\bar{M}$ is a strong symplectic filling of a contact manifold $(C, \xi)$ if $\bar{M}$ admits a Liouville form $\lambda$ defined in a collar neighbourhood of the boundary $\partial \bar{M}$ such that
\begin{enumerate}
\item the restriction $\left.\lambda\right|_{\partial \bar{M}}$ is a contact form that induces the boundary orientation on $\partial\bar{M}$
\item the contact manifolds $(\partial\bar{M}, \ker \left.\lambda\right|_{\partial \bar{M}})$ and $(C,\xi)$ are contactomorphic.
\end{enumerate}

\begin{theorem}\label{t:conformal}
    Let $(C,\xi)$ be a closed contact manifold with a contact form $\alpha$. If $(C,\alpha)$ admits a strong symplectic filling $\bar{M}$ (whose completion is denoted by $M$) with $SH^*_{M}(\{r_i\}\times\partial \bar{M})\neq 0$ for some $r_i\to\infty$, then any contact involutive map witnessed by $\alpha$ on $C$ has a fiber which is not contact displaceable in $C$. In fact, there exists a fiber $K$ such that for any contactomorphism $\phi: C\to C$ isotopic to the identity, we can find $p\in K$ with conformal factor $1$ with respect to $\phi$ and $\alpha$ so that $\phi(p)\in K$.
\end{theorem}

\begin{corollary}\label{c:contact big fiber}
    Let $(C,\xi)$ be a closed contact manifold, which is the ideal contact boundary of a finite type complete Liouville manifold $(M,\lambda).$ Assume that $SH^*(M)\neq 0$. Then, any contact involutive map $F: C\to \RR^N$ has a fiber which is not contact displaceable.
\end{corollary}
\begin{proof} 
Since $M$ is Liouville, the whole symplectization of $C$ can be embedded in $M$. Set $\bar{M}_i:=M\setminus((i,+\infty)\times C)$. The symplectic completion of every $\bar{M}_i$ is isomorphic to $M$. Therefore $SH^*(M)\neq 0$ implies that $SH^*_M(\bar{M}_i)\neq 0$. By Corollary \ref{c:boundary vanishing} we know $SH^*_M(\partial\bar{M}_i)\neq 0$, then we apply the above theorem.
\end{proof}

\begin{remark}
    The condition on the involutive map in Corollary~\ref{c:contact big fiber} cannot be relaxed by dropping the Liouville coordinate $r_\alpha$. Namely, there are maps $G=(g_1, \ldots, g_N): C\to \RR^N$ defined on the boundary of a Liouville domain with non-vanishing symplectic cohomology such that 
\[(g_1\circ\pi, \ldots, g_N\circ\pi): S(C)\to\RR^N\]
is an involutive map on $S(C)$ and such that all the fibers of $G$ are contact displaceable. One such example is a non-constant function $G:S^1\to\RR$ on the circle. Every fiber of $G$ is contact displaceable despite $S^1$ being the boundary of a Liouville domain (e.g. a compact genus-1 surface with 1 boundary component) with non-vanishing symplectic cohomology.
\end{remark}

A direct consequence of Corollary \ref{c:contact big fiber} is that the standard contact sphere $(S^{2n+1},\xi_{st})$ cannot bound any Liouville domain with non-zero symplectic cohomology, since it admits a contact involutive map with every fiber being contact displaceable \cite[Theorem 1.2]{MP}. This gives an alternative proof of  \cite[Corollary 6.5]{Se}. 

\subsection{Big fiber theorem for prequantization circle bundles}

An interesting family of examples comes from prequantization bundles. Let $(D,\omega_D)$ be a closed symplectic manifold with an integral lift $\sigma$ of its symplectic class $[\omega_D]$. Then there is a principle circle bundle $p: C\to D$ with Euler class $-\sigma$. One can choose a connection one-form $\alpha$ on $C$ such that $d\alpha=p^*\omega_D$, which makes $(C,\xi:=\ker\alpha)$ into a contact manifold, called a {prequantization bundle}. An important point for us is that involutive maps on $D$ pull back to contact involutive maps on $C$. If a compact subset $K$ of $D$ is displaceable in $D$ then $p^{-1}(K)$ is contact displaceable in $C$.

\begin{corollary}\label{c:prequan}
    Let $p:C\to D$ be a prequantization bundle over a closed symplectic manifold $D$. If $C$ admits a Liouville filling with non-zero symplectic cohomology, then for any involutive map $F:D\to \RR^N$ there is a contact nondisplaceable fiber of $F\circ p$.
\end{corollary}

Lots of examples satisfying Corollary \ref{c:prequan} come from the complement of a Donaldson hypersurface in a closed symplectic manifold. On the other hand, many prequantization bundles do not admit exact fillings. For example, $\RR P^{2n-1}$ with the standard contact structure is a prequantization bundle over $\CC P^{n-1}$ with no exact fillings when $n\geq 3$, see \cite[Theorem B]{Zhou}. However, it is well-known that sometimes one can do Floer theory on the symplectization of a contact manifold without a filling. Using this technique we prove the following theorem.

\begin{theorem}\label{thm-prequantum}
    Let $D$ be a closed positively monotone symplectic manifold with minimal Chern number at least $2$ and let $C$ be a prequantization bundle over $D$ with Euler class $-\sigma\in H^2(D;\ZZ)$, where $\sigma$ is an integral lift of $[\omega_D]$. If the quantum multiplication with $\sigma$ 
    $$
    I_{\sigma}:H^*(D; \Bbbk)[T, T^{-1}]\to H^*(D; \Bbbk)[T, T^{-1}], \quad A\mapsto A* \sigma
    $$
    is not a bijection, then for any involutive map on $D$ the pullback contact involutive map on $C$ has a contact non-displaceable fiber. 
\end{theorem}

This generalizes the result of Borman-Zapolsky \cite{BZ} and Granja-Karshon-Pabiniak-Sandon \cite{GKPS}. The main tool for the proof is the Floer-Gysin exact sequence of Bae-Kang-Kim \cite[Corollary 1.8]{BKK}. An important special case is when $\sigma$ is not a primitive class in $H^2(D;\ZZ)$. 

\begin{remark}
    We believe that the positively monotone and the minimal Chern number assumptions can be removed from the result.
\end{remark}

Let $p: (C,\xi)\to D$ be a prequantization bundle. Assume that $(D,\omega)$ is a smooth complex Fano variety equipped with a K\"ahler form. Using toric degenerations and the symplectic parallel transport technique, we can find many natural involutive maps $\pi: D\to B$ with non-empty fibers generically Lagrangian \cite{HK, Pr, KKP, GM}. Entov-Polterovich's big fiber theorem shows that $\pi$ needs to admit at least one Hamiltonian non-displaceable fiber. It is possible that for some $b\in B$, the fiber $\pi^{-1}(b)$ is Hamiltonian non-displaceable, but $(\pi\circ p)^{-1}(b)$ is contact displaceable. For example, the Clifford torus is non-displaceable in $\CC P^n$ while its preimage is contact displaceable in $(S^{2n+1},\xi_{st})$. Nevertheless, under the mild assumptions of Theorem \ref{thm-prequantum}, we show that $\pi\circ p$ also needs to have at least one contact non-displaceable fiber.

Although we did not include it in the statement for brevity, the conformal factor property still holds for an appropriate contact form in Theorem \ref{thm-prequantum}. Using this and an elementary geometric argument, we can also find applications of these results to the non-displaceability of Legendrian lifts of Lagrangians in $D$ from their Reeb closures in the style of Givental \cite{Giv}, Eliashberg-Hofer-Salamon \cite{EHS} and Ono \cite{Ono}. Our result applies more generally than these results but the conclusion is weaker in the sense that we do not have an estimate on the number of intersection points (c.f. \cite[Theorem 1.11]{BZ}).

\begin{theorem}\label{thm-elem-cont}
    Let $p: (C,\xi)\to D$ be a prequantization bundle with connection $1$-form $\alpha$. Let $L\subset D$ be a Lagrangian submanifold with a Legendrian lift $Z\subset C$ of $L$. That is, a Legendrian submanifold $Z$ of $C$ so that $p(Z)=L$ and $p\mid_Z:Z\to L$ is a finite covering map. If $\{1\}\times p^{-1}(L)\subset S(C)\cong (0,\infty)\times C$ is not Hamiltonian displaceable from itself inside $S(C)$, then $Z$ cannot be contactly displaced from $p^{-1}(L).$
\end{theorem}

Let us just note the following as a sample corollary, originally due to Givental \cite{Giv}.

\begin{corollary}
    Let $(\mathbb{R}{P}^{2n-1},\xi_{st})\to (\mathbb{C}{P}^{n},2\omega_{FS})$ be our prequantization bundle. Then, a Legendrian lift of the Clifford torus cannot be contactly displaced from the preimage of the Clifford torus.
\end{corollary}
\begin{proof}
Consider the standard toric structure on $\CC P^n$. The Clifford torus $L$ is a stem of this fibration. Our method shows that $\{1\}\times p^{-1}(L)\subset S(C)\cong (0,\infty)\times C$ is not Hamiltonian displaceable because symplectic cohomology with support at $\{1\}\times C$ is non-vanishing. Now we apply Theorem \ref{thm-elem-cont}.
\end{proof}

\subsection*{Acknowledgments}
Y.S. is supported by the EPSRC grant EP/W015889/1. I.U. is partially supported by the Science Fund of the Republic of Serbia, grant no. 7749891, Graphical Languages - GWORDS. U.V. is partially supported by Science Academy’s Young Scientist Awards Program (BAGEP) in Turkey. We express our sincere gratitude to the referee for their thorough review.

\section{Applications}\label{sec:app}

\subsection{Ustilovsky spheres}

Now we explain a concrete family of examples. Consider the Brieskorn variety
$$
M(p,2,\ldots,2):= \{z\in \CC^{n+1}\mid z_0^p+z_1^2+\cdots+ z_n^2=\varepsilon \},
$$
where $z_0,\ldots, z_n$ are the coordinates of $z$ and $n,p\in\NN$ are natural numbers. For small $\varepsilon>0$, it admits a Liouville structure such that its ideal boundary is the Brieskorn manifold
\[\Sigma(p,2,\ldots, 2):=\left\{z\in\CC^{n+1}\:|\: \abs{z}=1\:\text{and}\: z_0^p+z_1^2+\cdots+ z_{n}^2=0\right\}.\]
There exists a non-commutative complete integrable system on $\Sigma(p,2,\cdots, 2)$ that we will describe shortly. Non-commutative integrable systems were introduced in contact geometry by Jovanovi\'c \cite{J} in analogy to non-commutative integrable systems in symplectic geometry \cite{N}, \cite{MF}. A complete non-commutative integrable system on a $(2n+1)$-dimensional contact manifold consists of a contact form and contact Hamiltonians $f_1,\ldots, f_{2n-r}$ such that they all Jacobi commute with constants and with $f_1,\ldots, f_r$ and such that the functions $f_1,\ldots, f_{2n-r}$ are independent\footnote{The functions $f_1,\ldots, f_k$ are independent at a point $p$ if the linear maps $df_1(p), \ldots, df_k(p)$ are linearly independent.} on a dense and open subset. The contact Hamiltonians $f_{r+1},\ldots, f_{2n-r}$ are not assumed to commute with each other. The extreme case where $r=n$ corresponds to complete contact integrable systems introduced by Banyaga and Molino in \cite{BM}.

An Ustilovsky sphere is the $(4m+1)$-dimensional Brieskorn manifold given by
\[\Sigma(p,2,\ldots, 2):=\left\{z\in\CC^{2m+2}\:|\: \abs{z}=1\:\text{and}\: z_0^p+z_1^2+\cdots+ z_{2m+1}^2=0\right\},\]
where $p\equiv \pm1\pmod{ 8}$. Each Ustilovsky sphere is diffeomorphic to the standard smooth sphere, however their contact structures are not standard and they are all mutually non-contactomorphic \cite{Ust}. In \cite{JJ}, Jovanovi\'c and Jovanovi\'c give an example of a complete non-commutative integrable system on Ustilovsky spheres using its standard contact form. This system consists of the following contact Hamiltonians:
\begin{align*}
    & g_j(z):=i(\overline{z}_{2j}z_{2j+1}-z_{2j}\overline{z}_{2j+1}),\\
    & h_j(z):= \abs{z_{2j}}^2 + \abs{z_{2j+1}}^2,\\
    & q_j(z):= i\left( \overline{z}_1^2z_{2j}^2 - z_1^2\overline{z}_{2j}^2 \right) + i\left( \overline{z}_1^2z_{2j+1}^2 - z_1^2\overline{z}_{2j+1}^2 \right),
\end{align*}
where $j=1,\ldots, m.$ The contact Hamiltonians $g_1,\ldots, g_m$ are the commutative part of the non-commutative integrable system: they commute with all the contact Hamiltonians $$g_1,\ldots, g_m, h_1,\ldots, h_{m}, q_1,\ldots, q_m.$$ In particular, $G=(g_1,\ldots, g_m, h_1)$ is an involutive map on the Ustilovsky sphere. The next corollary applies Corollary~\ref{c:contact big fiber} to the involutive map $G$ proving that $G$ has a non-displaceable fiber. By work of Whitney \cite{W}, the fibers of $G$ are stratified. The maximal dimension of the stratum is equal to $3m$.

\begin{corollary}\label{c:ustilovsky}
    Let $C:=\Sigma(p,2,\ldots,  2)$ be an Ustilovsky sphere of dimension $4m+1$. Then, $C$ has a codimension-$(m+1)$ closed stratified submanifold that is contact non-displaceable.
\end{corollary}
\begin{proof}
    The Ustilovsky sphere is the ideal boundary of the Brieskorn variety 
    \[M:=\left\{z\in\CC^{2m+2}\:|\: z_0^p+z_1^2+\cdots+ z_{2m+1}^2=\varepsilon\right\}\]
    with non-vanishing symplectic cohomology $SH^*(M)\not=0$ \cite{KK}. By Proposition~4.4 in \cite{JJ}, the smooth map 
    \[G=(g_1, \ldots, g_m, h_1):C\to\RR^m,\]
    where $g_j(z):= i(\bar{z}_{2j}z_{2j+1}- z_{2j}\bar{z}_{2j+1})$ and $h_1=\abs{z_{2}}^2 + \abs{z_{3}}^2$, is involutive. Hence, Corollary~\ref{c:contact big fiber} implies that there exists a non-displaceable fiber of $G$, which is a stratified submanifold of maximal dimension $3m$.
\end{proof}
The non-displaceability of Corollary~\ref{c:ustilovsky} cannot be proved using only smooth topological methods. Indeed, Ustilovsky sphere $C$ is diffeomorphic to the standard sphere, and every non-dense subset of the standard sphere can be smoothly displaced. Corollary~\ref{c:ustilovsky} implies in particular Theorem~1.2 from \cite{U} which states that each Ustilovsky sphere admits two smoothly embedded closed balls such that one of them cannot be contactly squeezed into the other. We now elaborate why \cite[Theorem~1.2]{U} follows from Corollary~\ref{c:ustilovsky}. Let $C$ be an Ustilovsky sphere. Corollary~\ref{c:ustilovsky} implies that there exists a compact proper subset $K$ of $C$ such that $K$ is contact non-displaceable. As a homotopy sphere (actually, the standard smooth sphere), $C$ admits a Morse function $f:C\to\mathbb{R}$ with precisely two critical points: the minimum and the maximum. Moreover, the Morse function $f$ can be chosen in such a way that the critical points are not contained in $K$. The sublevel sets $f^{-1}((-\infty, c])$ for a value $c$ between the minimum and the maximum are smoothly embedded closed balls. For $c$ sufficiently close to the maximum, we have $K\subset f^{-1}((-\infty, c]).$ And, for $c$ sufficiently close to the minimum, the set $f^{-1}((-\infty, c])$ is disjoint from $K$. Since $K$ is contact non-displaceable, this means that sufficiently big (but not equal to $C$) sublevel set cannot be contactly squeezed into sufficiently small (but having more than 1 element) sublevel set of $f$. This proves \cite[Theorem~1.2]{U} because these sublevel sets are smoothly embedded closed balls.

\begin{remark}
    The reader surely has noticed that we only used the commutative part $g_1,\ldots, g_m$ of the non-commutative contact integrable system $g_1,\ldots, g_m, h_1,\ldots, h_{m}, q_1,\ldots, q_m$ together with an additional integral $h_1$. We think it is an interesting question whether one can improve the codimension of the rigid subset by creating a larger set of commuting functions from the integrals of this non-commutative system. This question of turning a non-commutative integrable system into a commutative one has been extensively studied \cite{BJ} but it seems difficult to get a bound on the codimension of the fibers using the constructions in the literature.
\end{remark}

\subsection{Quadrics}\label{sec:quadrics}

We will focus our attention on the example of quadrics to illustrate the prequantization bundle case. Let $Q_n$ be the complex $n$-dimensional quadric hypersurfaces in $\CC P^{n+1}$, with the induced Fubini-Study symplectic structure. It admits a Donaldson hypersurface, symplectomorphic to $Q_{n-1}$ up to scaling, such that $Q_n-Q_{n-1}$ is symplectomorphic to a disk subbundle of $T^*S^n$. Here $S^n$ is the standard $n$-dimensional sphere and $T^*S^n$ is equipped with the standard sympeclectic form. Therefore $T^*S^n$ admits an ideal contact boundary which is a prequantization bundle over $Q_{n-1}$. 

In general, $Q_n$ admits a Gelfand-Zeitlin integrable system whose moment polytope is a simplex in $\RR^n$, determined by
\begin{equation}\label{eq:GZ}
    n\geq u_n\geq \cdots \geq u_2\geq \abs{u_1}.
\end{equation}
See \cite[Section 2.2]{Kim} for a nice summary of construction.

It is known that $Q_2$ is symplectomorphic to $(S^2\times S^2,\sigma\times\sigma)$ where the two factors both have the same area. Hence there exist two integrable systems on $Q_2$. The first one is the standard toric system, whose moment polytope is a square. The second one is the Gelfand–Zeitlin system, whose moment polytope is a triangle, see Figure \ref{fig:Q2}.

\begin{figure}
\begin{tikzpicture}[scale=2]

\draw (0,0)--(1,0)--(1,1)--(0,1)--(0,0);
\draw [blue,fill=blue] (0.5,0.5) circle (1pt);

\draw (3,1)--(5,1)--(4,0)--(3,1);
\draw [blue] (4,0.5)--(4,0);
\draw [blue,fill=blue] (4,0.5) circle (1pt);
\draw [blue,fill=blue] (4,0) circle (1pt);

\end{tikzpicture}
\caption{Two moment polytopes for $Q_2$. On the left, the dot represents the barycenter. On the right, the line segment is contained in the median and the interior dot is at the midpoint of this median.}\label{fig:Q2}
\end{figure}
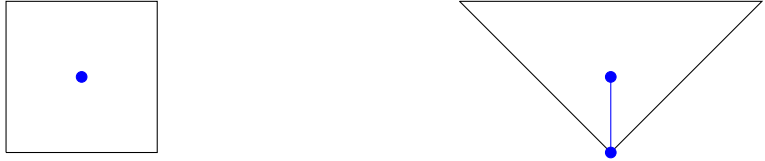

We recall some results about the two fibrations on $Q_2$, using Figure \ref{fig:Q2}.
\begin{enumerate}
    \item The central fiber in the square is a Lagrangian torus, called the Clifford torus, which is a stem in this fibration, by McDuff \cite[Theorem 1.1]{McDuff}.
    \item The fiber over the interior dot in the triangle is a Lagrangian torus, called the Chekanov torus. The fiber over the bottom vertex in the triangle is a Lagrangian sphere, which is the anti-diagonal.
    \item Any fiber over the blue segment the triangle, including the ends, is non-displaceable by Fukaya-Oh-Ohta-Ono \cite[Theorem 1.1]{FOOO}. Other fibers are displaceable.
\end{enumerate}

\begin{figure}
\begin{tikzpicture}

\draw (0,0)--(2,1)--(2,-3)--(-1,-1)--(2,1);
\draw [red] (0,0)--(-1,-1);
\draw [dotted] (0,0)--(2,-3);

\draw [->] (0,0)--(0,1.5);
\draw [->] (2.2,0)--(3,0);
\draw [dotted] (0,0)--(2.2,0);
\draw [dotted] (-1,-1)--(-1.2,-1.2);
\draw [->] (-1.2,-1.2)--(-1.8,-1.8);

\node [right] at (0,1.6) {$u_1$};
\node [right] at (3,0) {$u_2$};
\node [below] at (-1.8,-1.8) {$u_3$};

\draw [red,fill=red] (0,0) circle (2pt);
\draw [blue] (-1,-1)--(0.5,-1);
\draw [blue,fill=blue] (0.5,-1) circle (2pt);

\draw [red] (6,1)--(6,-2);
\draw (6,1)--(9,-2);
\draw [red,fill=red] (6,1) circle (2pt);
\draw [->, dotted] (7,0)--(7,-1.5);
\draw (7,0) node {$\times$};
\draw (7,-0.8) node {$\times$};
\node [right] at (7,0.2) {$w$};
\node [left] at (6.8,-0.8) {$x$};
\node [below] at (7,-1.6) {$p_v(w)$};
\node [below] at (7,-2.5) {$\{u_1=0\}$-plane};

\end{tikzpicture}
\caption{Moment polytope for $Q_3$ and a probe in $T^*S^3$. The vertices of the polytope on the left are $(0,0,0),(0,0,3),(\pm 3,3,3).$}\label{fig:Q3}
\end{figure}
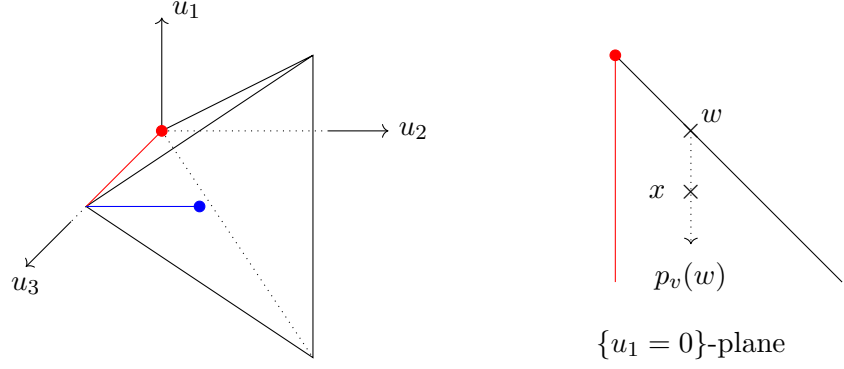

Using equation (\ref{eq:GZ}) for $n=3$, we get a moment polytope for $Q_3$, depicted in Figure \ref{fig:Q3}. One can think of this polytope as a cone over the (rotated) triangle in Figure \ref{fig:Q2}, which corresponds to the face $\{u_3=3\}$. If we remove this face, we get a disk subbundle of $T^*S^3$.  

\begin{lemma}\label{l:Chekanov}
    The whole $T^*S^3$ admits an integrable system whose polytope is $P:=\{u_3\geq u_2\geq \abs{u_1}\}$ in $\RR^3$. It satisfies that
    \begin{enumerate}
        \item All fibers are compact smooth isotropic submanifolds.
        \item The fiber over $(0,0,0)$ is the Lagrangian zero section $S^3$, red dot in Figure \ref{fig:Q3}.
        \item The fiber over $(0,0,\lambda>0)$ is a Lagrangian $S^1\times S^2$, red line in Figure \ref{fig:Q3}.
        \item Any fiber that is not over $(0,0,\lambda\geq 0)$ is Hamiltonian displaceable in $T^*S^3$.
    \end{enumerate}
\end{lemma}
\begin{proof}
    The first three items are in \cite[Example 2.4]{Kim}. We will prove $(4)$ by using McDuff's probe. In the following we freely use the notions from \cite[Section 2]{McDuff}.
    
    It suffices to consider fibers with $u_1=0$. Otherwise they are preimages of displaceable fibers in $Q_2$ under the prequantization bundle map. Pick a point $x=(0,a>0,b)$ in $P$, we have $w:=(0,a,a)$ in the interior of the facet $F:=\{u_2=u_3\}$ and $x$ is on the probe $p_v(w)$ with direction $v:=(0,0,1)$. Two vectors $(0,1,1)$ and $(1,0,0)$ are parallel to $F$ hence $v$ is integrally transverse to $F$. Then \cite[Lemma 2.4]{McDuff} says that the fiber over $x$ is displaceable. Note that this Lemma assumes that the integrable system is toric. However its proof is local and our probe $p_v(w)$ lives totally in the toric part. See Figure \ref{fig:Q3} which shows the probe in the $u_2u_3$-plane.
\end{proof}

\begin{remark}
\begin{enumerate}
    \item In the above we used that our probe $p_v(w)$ has infinite length hence the distance between $x$ and $w$ is always less than half of the length of the probe. This is not true in the compact $Q_3$ picture, since certain preimage of the Chekanov torus is indeed non-displaceable in $Q_3$, by \cite[Corollary 4.3]{Kim}.
    \item The probe method shows that the preimage of the Chekanov torus is Hamiltonian displaceable in $T^*S^3$. The Hamiltonian isotopy pushes it away to infinity. A natural question is: is it contact displaceable within the ideal boundary?
\end{enumerate}

\end{remark}

Let $E_T$ be the preimage of the Clifford torus in the ideal boundary $C$ of $T^*S^3$ and $E_S$ of the anti-diagonal. Identify $T^*S^3$ with $S\cup (\RR_+\times C)$ where $S$ is the zero section. If we view $\{1\}\times C$ as the preimage of the triangle $\{u_3=1\}$ in the moment polytope in Lemma \ref{l:Chekanov}, then $E_S$ is the fiber over $(0,0,1)$. For $E_T$, we will use another polytope for $T^*S^3$, which is a cone over the square on the left side of Figure \ref{fig:Q2}.

\begin{corollary}
    $E_T$ and $E_S$ are both contact non-displaceable in the ideal boundary $C$ of $T^*S^3$.
\end{corollary}
\begin{proof}
    Suppose that $E_S$ is contact displaceable in $C$, then $E_S$ is Hamiltonian displaceable in $T^*S^3$, viewed as a subset of $\{1\}\times C\subset T^*S^3$. Combined with $(4)$ of Lemma \ref{l:Chekanov}, we get a finite Poisson commuting cover of $\{1\}\times C\subset T^*S^3$ with Hamiltonian displaceable elements. Here we view $\{1\}\times C$ as the preimage of the triangle $\{u_3=1\}$ in the moment polytope of $T^*S^3$. Our assumption implies that any fiber living above this triangle is Hamiltonian displaceable. Therefore, any point in this triangle admits a small neighborhood in the moment polytope whose preimage is Hamiltonian displaceable in $T^*S^3$. The preimages of these neighborhoods give us a Poisson commuting cover of $\{1\}\times C\subset T^*S^3$, and we pick a finite subcover by compactness. Therefore $SH^*_{T^*S^3}(\{1\}\times C)=0$, which contradicts Corollary \ref{c:boundary vanishing} and $SH^*(T^*S^3)\neq 0$. The same proof works for $E_T$.
\end{proof}

\subsubsection{Non-compact skeleta}

We now aim to obtain another perspective on the intersection result of \cite[Theorem 1.2]{AD} when $n=3$. 

\begin{corollary}\label{c:EtEc}
For any compact subset $K$ of $T^*S^3$ which is disjoint from $S\cup(\RR_+\times E_T)$ or $S\cup(\RR_+\times E_S)$, we have that $SH^*_M(K)=0$.

\end{corollary}
\begin{proof}
    Let $\pi:\RR_+\times C\to C$ be the projection and $p:C\to S^2\times S^2$ be the prequantization map. For a compact subset $K$ of $T^*S^3$ which is disjoint from $S\cup(\RR_+\times E_T)$, we have $p\circ\pi(K)$ disjoint from the Clifford torus. The Clifford torus is a stem in the product toric fibration of $S^2\times S^2$, meaning that all other fibers are displaceable \cite[Theorem 1.1]{McDuff}. Therefore $p\circ\pi(K)$ is covered by finitely many Poisson commuting sets $B_i$ in $S^2\times S^2$, with each $B_i$ displaceable. Then the sets $[a_i,b_i]\times p^{-1}(B_i)$ are displaceable and Poisson commuting in $T^*S^3$, see Section \ref{sec:conformal}. In fact, the sets $[a_i,b_i]\times p^{-1}(B_i)$ are displaceable in the complement of the zero section, which can be identified with the symplectization $\mathbb{R}_+\times C$. Choosing $a_i$ small and $b_i$ large, the sets $[a_i,b_i]\times p^{-1}(B_i)$ cover $K$ and we can apply Mayer-Vietoris. We finish with the familiar unitality of restriction maps argument \cite{TVar}, whose details we omit. The same proof works for $S\cup(\RR_+\times E_S)$, with the help of Lemma \ref{l:Chekanov}.
\end{proof}

A closed symplectic manifold $M$ usually admits some skeleton which intersects all Floer-essential Lagrangians, see \cite[Corollary 1.25]{TVar} and \cite[Theorem D]{BSV}. In parallel languages, there is the notion of superheavy set \cite{EP09}, which intersects all heavy sets. The above set $S\cup (\RR_+\times E)$ serves as a non-compact analogue when $M$ is Liouville.

As mentioned before, there is an integrable system on $Q_n$ for any dimension $n$. We expect that a similar computation of probes as in Lemma \ref{l:Chekanov} would recover \cite[Theorem 1.2]{AD} in all dimensions. Moreover, by using a covering trick, all these results should have analogues for $T^*\RR P^n$.

\subsection{Contact non-squeezing}

Another way of looking at contact non-displaceability is via \emph{contact non-squeezing}. The notion of contact non-squeezing was introduced by Eliashberg-Kim-Polterovich \cite{EKP} although the first instance of a (genuine) contact non-squeezing was proven in an earlier paper by Eliashberg \cite{E}. Here is the definition of contact non-squeezing.
\begin{definition}
    Let $(C,\xi)$ be a contact manifold. Then, a subset $\Omega_1\subset C$ can be contactly squeezed inside a subset $\Omega_2\subset C$ if there exists a compactly supported contact isotopy $\phi_t:C\to C, t\in[0,1]$ such that $\phi_0=\operatorname{id}$ and such that $\phi_1(\bar{\Omega}_1)\subset \Omega_2$.
\end{definition}

Now, contact displaceability can be rephrased as follows: a compact subset $K$ of a contact manifold $(C,\xi)$ can be contactly displaced if, and only if, $K$ can be contactly squeezed into its complement $C\setminus K$. By applying Corollary~\ref{c:contact big fiber} to the case of an involutive map consisting of a single contact Hamiltonian, we obtain the following corollary.

\begin{corollary}\label{c:nonsq}
    Let $(C,\xi)$ be a closed contact manifold that is fillable by a Liouville domain with non-zero symplectic cohomology. Let $h:C\to\RR$ be a contact Hamiltonian that is invariant under the Reeb flow with respect to some contact form. Then, one of the sets $\{h\geqslant 0\}$ and $\{h\leqslant0\}$ cannot be contactly squeezed into its complement.
\end{corollary}
\begin{proof} Corollary~\ref{c:contact big fiber} implies that $h$ has a contactly non-displaceable fiber. This fiber belongs either to $\{h\geqslant 0\} $ or to $\{{h}\leqslant 0 \}$. This finishes the proof.
\end{proof}

\begin{remark}
    We note that the decomposition of the contact manifold as the union of $\{h\geqslant 0\}$ and $\{h\leqslant0\}$ depends only on the contact vector field of $h$. 
\end{remark}

Combined with prequantizations, a simple example can be obtained.

\begin{example}
    Let $p:(C,\xi,\alpha)\to D$ be a prequantization bundle. Let $F:D\to \RR$ be a function such that $\{F\leq 0\}$ is displaceable in $D$. If $C$ admits a Liouville filling with non-zero symplectic cohomology or satisfies Theorem \ref{thm-prequantum}, then $\{F\circ p\geq 0\}\subset C$ cannot be contactly squeezed into its complement.
\end{example}

Particularly interesting is the case of Corollary~\ref{c:nonsq} where $0$ is a regular value of $h$ and where $\{h\geqslant 0\}$ and $\{h\leqslant0\}$ are contact isotopic. In this case, $\{h> 0\}$ is an example of a set that can be smoothly squeezed into itself but not contactly. 

\begin{example}
    Let $X$ be a smooth vector field on a closed smooth manifold $P$ such that $X$ gives rise to an $\mathbb{S}^1$ action on $P$. Let $h:S^\ast P\to\RR$ be the contact Hamiltonian on the unit cotangent bundle defined by $h(v^\ast):=v^\ast(X\circ\pi)$ where $\pi:T^\ast P\to P$ is the canonical projection. Then, the set $\{h\geqslant 0\}$ is contact non-displaceable. In other words, the set $\{h\geqslant 0\}$ cannot be contactly squeezed into its complement.
\end{example}
\begin{proof}
    The contact Hamiltonian $h$ generates contact circle action $\varphi:S^\ast P\to S^\ast P$ that is the `lift' of the given circle action on $P$. By averaging (see \cite[Proposition~2.8]{L} and \cite[Lemma~3.4]{DU}), we can prove that there exists a contact form $\alpha$ on $S^\ast P$ such that $\varphi_t^\ast\alpha=\alpha$ for all $t$. Denote by $f:S^\ast P\to\RR$ the contact Hamiltonian of $\varphi_t$ with respect to $\alpha.$ Notice that $f$ is Reeb invariant with respect to $\alpha$ and that the sets $\{h\geqslant0\}$ and $\{f\geqslant 0\}$ are equal. Since $SH^*(T^\ast P, \mathbb{Z}_2)\not=0$, Corollary ~\ref{c:nonsq} implies that one of the sets $\{h\geqslant 0\}=\{f\geqslant 0\}$ and $\{h\leqslant 0\}=\{f\leqslant 0\}$ cannot be contactly squeezed into its complement. Let $a:S^\ast P\to S^\ast P$ be the map given by $a(v^\ast)=-v^
    \ast$. The map is a contactomorphism (not preserving the coorientation) that satisfies $h\circ a=-h$. As a consequence (of the existence of $a$), the set $\{h\geqslant 0\}$ can be contactly squeezed into its complement if, and only if, the same is true of the set $\{h\leqslant0\}$. Therefore, neither of the sets $\{h\geqslant 0\}$ or $\{h\leqslant 0\}$ can be contactly squeezed into its complement. This finishes the proof.
\end{proof}

\section{Proofs}\label{sec:proofs}
In this section we prove Theorem \ref{thm-rfh-sh-vanish}, Theorem \ref{t:conformal} and Theorem \ref{thm-prequantum}. They are all consequences of the Mayer-Vietoris principle for the symplectic cohomology with support, generalized to the corresponding settings. Since its introduction in \cite{Var21}, symplectic cohomology with support has been further developed in \cite{TVar,BSV,GV2,AGV}. Now we briefly recall its definition to fix the notation. The version that we use in this article is closest to the one in \cite{BSV}.
\subsection{The definition of symplectic cohomology with support}\label{sec:def-sh}

Let $M$ be a symplectic manifold that is convex at infinity. We fix a decomposition 
$$
M:= \bar{M}\cup_{\partial\bar{M}} ([1,+\infty)\times \partial\bar{M})
$$
where $\bar{M}$ is a compact symplectic manifold with contact boundary and $[1,+\infty)\times \partial\bar{M}$ is a cylindrical end with a Liouville coordinate. 

Let $A$ be the quotient of the image of $\pi_2(M)\to H_2(M)$ by the subgroup of classes $a$ such that $\omega(a)=0$ and $2c_1(TM)(a)=0$. We define $\Lambda$ be the degree-wise completion of the $\mathbb{Z}$-graded and valued group ring $\Bbbk[A]$ where $e^{a}$ has valuation $\omega(a)$ and grading $2c_1(TM)(a)$. We do not make any restrictions on $\Bbbk$ if $c_1(TM)$ and $\omega$ are proportional to each other with non-negative constants when restricted to $\pi_2(M)$. For a general symplectic manifold where virtual techniques are needed to define Hamiltonian Floer cohomology, we require that $\Bbbk$ contains $\QQ$.

Let $\gamma:S^1 \to M$ be a nullhomotopic loop in $M$. 
A \emph{cap} for $\gamma$ is an equivalence class of disks $u:\mathbb{D} \to M$ bounding $\gamma$, 
where $u \sim u'$ if and only if the Chern number and the symplectic area of the spherical class $[u-u']$ vanishes. 
The set of caps for $\gamma$ is a torsor for $A$.

Given a non-degenerate Hamiltonian $H: S^1 \times M \to \mathbb{R}$. For a capped orbit $\tilde{\gamma} = (\gamma, u)$ we have a $\mathbb{Z}$-grading 
and an action
$$
	i(\gamma, u) = \text{CZ}(\gamma, u) + \frac{\dim(M)}{2} \quad\mbox{and}\quad
	\mathcal{A}(\gamma, u) := \int_{S^1} H(t, \gamma(t))\,dt + \int u^*\omega, 
$$
and these are compatible with the action of $A$ in that
$$
	i(a \cdot (\gamma, u)) = i(\gamma, u) + 2c_1(TM)(a)
	\qquad\text{and}\qquad
	\mathcal{A}(a \cdot (\gamma, u)) = \mathcal{A}(\gamma, u) + \omega(a)\,.
$$

Since $M$ is non-compact, we need extra conditions on the Hamiltonians and almost complex structures to control Floer solutions. Here let us use the most common choice: Hamiltonians are linear near infinity (up to a constant) and almost complex structures that are of contact type near infinity for our fixed choice of convex cylindrical end, see \cite[Section 3.3]{Se}.

Define $CF^*(H)$ to be the free $\mathbb{Z}$-graded $\Bbbk$-module generated by the capped orbits.
It is naturally a graded $\Lambda$-module, via $e^a \cdot (\gamma,u) := a \cdot (\gamma,u)$. It also admits a Floer differential after the choice of a generic $S^1$-family of $\omega$-compatible almost complex structures (which we suppress from the notation). The differential is $\Lambda$-linear, increases the grading by $1$, does not decrease action, and squares to zero.
One can also define continuation maps $CF^*(H_0) \to CF^*(H_1)$ in the standard way. If the continuation maps are defined using monotone Floer data, then the continuation map $CF^*(M, H_0) \to CF^*(M, H_1)$ does not decrease action.

For a compact subset $K$ of $M$, one can use the following acceleration data to compute the symplectic cohomology with support on $K$:
\begin{itemize}
    \item A sequence of non-degenerate Hamiltonian functions $H_n$ which monotonically approximate from below $\chi_K^\infty$, the lower semi-continuous function which is zero on $K$ and positive infinity outside $K$.
    \item A monotone homotopy of Hamiltonian functions connecting $H_n$'s.
    \item A suitable family of almost complex structures.
\end{itemize}

Given the above acceleration data, 
we define the chain complex
$$
tel^*(\mathcal{C}):=\bigoplus_{n=1}^\infty(CF^*(H_{n})\oplus CF^*(H_{n})[1])
$$ 
by using the telescope construction. The differential $\delta$ is defined as follows, if $x_{n}\in CF^k(H_{n})$ then
$$
\delta x_{n}= d_{n}x_{n} \in CF^{k+1}(H_{n}), 
$$
and if $x'_{n}\in CF^k(H_{n})[1]$ then

\begin{align}\label{eq:delta}
	\delta x'_{n}&= (x'_{n}, -d_{n}x'_{n}, -h_{n}x'_{n})\\
	&\in CF^{k}(H_{n})\oplus CF^{k+1}(H_{n})[1]\oplus CF^k(H_{n+1}).
\end{align}
Here $d_n: CF^k(H_n)\to CF^{k+1}(H_n)$ is the Floer differential and $h_n: CF^k(H_n)\to CF^{k}(H_{n+1})$ is the continuation map. Then we take a degree-wise completion in the following way. The Floer complex $CF^*(H)$ is equipped with an action filtration. Every element in $tel^*(\mathcal{C})$ is a finite sum of elements from $CF^*(H_n)$'s. We define the action of such a sum as the smallest action among its summand, and call it the $\min$-action on $tel^*(\mathcal{C})$. For any number $a\in [-\infty, \infty)$, we have subcomplexes $CF^*_{\geq a}(H_n)$ and $tel^*(\mathcal{C})_{\geq a}$ containing elements with action greater or equal to $a$. For $a<b$ we form the quotients
$$
CF^*_{[a,b)}(H_n):=CF^*_{\geq a}(H_n)/CF^*_{\geq b}(H_n), \quad tel^*(\mathcal{C})_{[a,b)}:= tel(\mathcal{C})_{\geq a}/tel(\mathcal{C})_{\geq b}.
$$
The degree-wise completion of the telescope is defined as
$$
\widehat{tel^k}(\mathcal{C}):= \varprojlim_b tel^k(\mathcal{C})_{(-\infty,b)}
$$
as $b\to +\infty$. The symplectic cohomology with support on $K$ can be computed as the homology of the completed telescope 
\begin{equation}\label{eq:defSH}
    SH^*_M(K)=H(\widehat{tel^*}(\mathcal{C})).
\end{equation} 
Using the monotonicity techniques developed by Groman, we can show that different choices of acceleration data give isomorphic homology (see \cite[Proposition 6.5]{GV} or \cite[Theorem 4.17]{GV2}). In fact, we can also show independence on the choice of our convex cylindrical end.

On the other hand, if we add the extra condition on the acceleration data that the slopes of $H_n$ go to infinity as $n$ goes to infinity (with respect to the fixed convex cylindrical end), then we can reach the same conclusion relying entirely on standard maximum principle arguments. In particular, the sandwich argument \cite[Proposition 3.3.3]{Var21} continues to work with the standard maximum principle, provided that all Hamiltonians and almost complex structures are chosen as described. The latter is sufficient for our purposes except at one small point (that is not actually used in the main body of the paper). In the Appendix, to compare our restriction maps with those of Viterbo, we will use Hamiltonians that are $C^2$ small with irrational slope at infinity. This is also for convenience only and could be avoided.

\begin{comment}
    \begin{remark}
    Our Floer complex $CF^*(H_n)$ is generated by contractible orbits. Each of them is equipped with a cap to define the action. The completion of the telescope is with respect to the action filtration. This is different from the original definition in \cite{Var21} which uses the Novikov filtration, but the same as \cite{BSV} and \cite[Section 4.3]{GV2}. When $M$ is a Liouville manifold, the action of a capped orbit does not depend on the cap hence the Floer complex is generated by orbits. 
\end{remark}

\end{comment}

In the next section we prove Theorem \ref{t:conformal} using this definition of symplectic cohomology with support. In Sections \ref{sec-pf-rfh-sh} and \ref{sec-quantum-proof}, we will extend the definition to other cases of $K\subset M$ to prove the desired results.

\subsection{Proof of Theorem \ref{t:conformal}}\label{sec:conformal}
Let $(C,\xi)$ be a closed contact manifold with a contact form $\alpha$. A strong symplectic filling $(\bar{M},\omega)$ of $(C,\xi,\alpha)$ gives us 
$
M:= \bar{M}\cup_{\partial\bar{M}} ([1,+\infty)\times \partial\bar{M}),
$ as in the previous section.
Now we recall the relation between contact displacement in $C$ and Hamiltonian displacement in the symplectization $S(C)=((0,+\infty)\times C, d(r\alpha))$. We write $R_\alpha$ as the Reeb vector field of $\alpha$. For a time-dependent function $h_t$ on $C$ there is a unique contact Hamiltonian vector field $V_{h_t}$ defined by
$$
\alpha(V_{h_t})=h_t, \quad d\alpha(V_{h_t},\cdot)= dh_t(R_\alpha)\alpha- dh_t.
$$
Consider the induced function $H=r h_t+c$ on the symplectization, where $c\in\mathbb{R}$ is a constant. The contact Hamiltonian vector field $V_h$ and the (symplectic) Hamiltonian vector field $X_H$ is related as
$$
X_H= -V_{h_t} + dh_t(R_\alpha)\cdot r\partial_{r}.
$$
Our convention for the Hamiltonian vector field is $dH= \omega(X_H, \cdot)$. Therefore if a compact subset $K$ of $C$ is contact displaceable, then the infinite block $\RR_+\times K$ is (Hamiltonian) displaceable in $S(C)$. Particularly, any finite block $[a,b]\times K$ is displaceable by a Hamiltonian compactly supported in $(1,+\infty)\times C$ when $a$ is large\footnote{Notice that we do not have freedom to choose $a$ arbitrarily small (as in the proof of Corollary~\ref{c:EtEc}) because the filling is not necessarily exact in Theorem~\ref{t:conformal}.}. The Hamiltonain isotopies that are used to displace $[a,b]\times K$ rescale the $\RR_+$ coordinate in the symplectization $\RR_+\times C$ (cf. Lemma~\ref{l:conformal}). Hence, large enough $a$ is required here to make sure the image of $[a,b]\times K$ under these Hamiltonian isotopies is always contained in $(1, +\infty)\times C$. Furthermore, a stronger displacement result holds.

\begin{lemma}\label{l:conformal}
    Fix a contact form $\alpha$ and the Liouville coordinate $r_\alpha$ on $S(C)$. Let $\phi_t$ be the flow of the contact Hamiltonian vector field $V_h$ for some $h$. Let $\Phi_t$ be the flow of the (symplectic) Hamiltonian vector field $X_H$ for $H=r_\alpha h+c$. Then for any $p\in\{r_\alpha=r_0\}\subset S(C)$, if $c_0:=\frac{r_\alpha(\Phi_1(p))}{r_0}$, we have $(\phi_1^*\alpha)_p=c_0\alpha_p$. 
\end{lemma}
\begin{proof}
    Since the flow $\Phi_t$ is $\RR_+$ equivariant, the number $c_0$ (actually, a function with variable $p$) does not change as $p$ moves along a (fixed) Liouville trajectory. Therefore (because each Liouville trajectory intersects the set $\{r_\alpha=1\}$), it is enough to prove the lemma for $r_0=1$.
    The component of $X_H$ in the Liouville direction is $dh(R_\alpha)r_\alpha$. Therefore
    $$
    r_\alpha(\Phi_1(p))= \exp\bigg(\int_0^1dh(R_\alpha)(\phi_t(p))dt\bigg).
    $$
    On the other hand, we can compute that
    $$
    \dfrac{d}{dt}(\phi_t^*\alpha)=\phi_t^*(d(\alpha(V_h))+ d\alpha(V_h,\cdot))=\phi_t^*(dh(R_\alpha)\alpha)
    $$
    which gives
    $$
    (\phi_1^*\alpha)_p= \alpha_p\cdot\exp\bigg(\int_0^1dh(R_\alpha)(\phi_t(p))dt\bigg).
    $$
\end{proof}

By this lemma, we know that if a compact subset $K\subset \{r_\alpha=1\}\subset S(C)$ is nondisplaceable from itself under $\Phi_1$ in $S(C)$, then there is a point $p\in K$ such that $\phi_1(p)\in K$ and $p$ has conformal factor one. The following lemma is also used in the proof of Theorem~\ref{t:conformal}.

\begin{lemma}\label{l:cont2symp}
Let $M$ be the completion of a strong filling of a contact manifold $C$. Let $\alpha$ be a contact form on $C$ and $((\varepsilon,+\infty)\times C, d(r\alpha))$ a convex end of $M$. Let $G:C\to\RR^N$ be a contact involutive map witnessed by $\alpha$. Then, the sets $[a_1,b_1]\times G^{-1}(K_1)$ and $[a_2,b_2]\times G^{-1}(K_2)$ Poisson commute in $M$ for all $[a_1,b_1], [a_2,b_2]\subset(\varepsilon, +\infty)$ and compact $K_1, K_2\subset \RR^N.$
\end{lemma}
\begin{proof}
The lemma essentially follows from Definition~\ref{d:contact Poisson}. Let $\mu:(\varepsilon, +\infty)\to [0,+\infty)$ be a smooth non-decreasing function such that $\mu(r)=0$ for $r$ near $\varepsilon$ and $\mu(r)=1$ for $r\geqslant \min(a_1,a_2)$. Since $G$ is a contact involutive map witnessed by $\alpha$, the map
\[ F:M\to \RR^{N+1} \quad:\quad p\mapsto \left\{ \begin{matrix}(r\cdot \mu(r), G(x)\cdot \mu(r))& \text{if }p=(r,x)\in (\varepsilon, +\infty)\times C\\ 0 &\text{otherwise}\end{matrix}\right. \]
is (symplectic) involutive in $M$. In addition, we have  $[a_i,b_i]\times G^{-1}(K_i)= F^{-1}([a_i,b_i]\times K_i)$. A special case is where $a_i=b_i$. 
\end{proof}

\begin{proof}
    [Proof of Theorem \ref{t:conformal}]
    Let $G:C\to \RR^N$ be a contact involutive map, with every fiber being contact displaceable. Since $C$ is compact, there is a finite open cover $\{U_i\}$ of $G(C)$ in $\RR^N$ such that each $G^{-1}(\bar{U}_i)$ is contact displaceable in $C$. We can view these sets as subsets of $\{r_i\}\times \partial\bar{M}$. For $r_i$ large enough, they are Hamiltonian displaceable in $(1, +\infty)\times \partial\bar{M}$, and therefore also in $M$. (Since the filling $\bar{M}$ is not necessarily exact, we make sure that contact isotopies that displace the sets $\{r_i\}\times G^{-1}(\bar{U}_i) $ have supports in $(1, +\infty)\times \partial\bar{M}$. This is the reason we require $r_i$ large enough.) Denote $r_M:=\max\{r_i\}$. Then, the sets $\{r_M\}\times G^{-1}(\bar{U}_i)$ form a Poisson commuting cover of $\{{r_M}\}\times\partial\bar{M}$ in $M$, by Definition \ref{d:contact Poisson} (see Lemma~\ref{l:cont2symp}). Therefore, the Mayer-Vietoris principle implies that $SH^*_M(\{r_M\}\times\partial\bar{M})=0$, which leads a contradiction. The claim about conformal factor follows from the Lemma~\ref{l:conformal}. Indeed, if for each fiber $K$ of $G$ there exists a contactomorphism $\phi:C\to C$ contact isotopic to the identity such that $(\phi^\ast\alpha)_p\not=\alpha_p$ for all $p\in K$ satisfying $\phi(p)\in K$, then $\{r\}\times K$ can be displaced in $M$ (but perhaps not contactly displaced in $\{r\}\times\partial\bar{M}$) for $r$ sufficiently large. Notice that for this part of the theorem (as opposed to the first part) the Liouville direction played a major role in displacing the fibers. Once we established that each fiber is displaceable in $M$, we can repeat the argument above and obtain a contradiction.
\end{proof}

\subsection{Proof of Theorem \ref{thm-rfh-sh-vanish}}\label{sec-pf-rfh-sh}

For a symplectic manifold $M$ with a fixed convex cylindrical end as above, write $K_M:=[1,+\infty)\times \partial\bar{M}$. We first define symplectic cohomology with support on $M$ or on $K_M$. Although these sets are non-compact, they are simple enough to extend the definition directly. 

Let $r$ be the Liouville coordinate on $K_M$ and $\alpha$ be the contact form on $\partial\bar{M}$. A Hamiltonian function $H$ on $M$ has a small negative slope at infinity if it is of the form $-\epsilon r$ outside a compact subset of $K_M \cup \bar{M}$, for some $\epsilon>0$ less than the minimal period of the Reeb orbits of the contact form $\alpha$. We define $SH^*_M(M)=H(\widehat{tel^*}(\mathcal{C}))$ where $\widehat{tel^*}(\mathcal{C})$ is the degree-wise completed telescope formed by Hamiltonian functions that are negative on $M$, i.e. everywhere, monotonically converging to zero on $M$, and have small negative slopes at infinity with slopes converging to zero. Similarly, we define $SH^*_M(K_M)$ using functions that are negative on $K_M$, converging to zero on $K_M$, diverging to positive infinity outside $K_M$, and have small negative slopes at infinity with slopes converging to zero.

Let $H^*_c(M;\Lambda)$ be the compactly supported cohomology of $M$, and let $H^*(M;\Lambda)$ be the cohomology of $M$. Both of them are $\Lambda$-modules, and can be equipped with a quantum product using genus zero three-pointed Gromov-Witten invariants \cite[Section 2.12]{Ritter2014}.

\begin{proposition}\label{p:compact support}
    There is a ring isomorphism between $SH^*_M(M)$ and $H^*_c(M;\Lambda)$ with respect to the quantum product.
\end{proposition}
\begin{proof}
    Pick a small negative Morse function $f$ on $M$ which is linear at infinity with a small negative slope. Then the sequence of functions $f_n:=f/n$ can be filled with monotone homotopy and made into a Floer one-ray $\mathcal{C}_f$ to compute $SH^*_M(M)=H(\widehat{tel^*}(\mathcal{C}_f))$. Following \cite{TVar,AGV}, one can equip $SH^*_M(M)$ with a product structure. It is isomorphic to the compactly supported cohomology $H^*_c(M)$ of $M$ equipped with the quantum product as a ring using the PSS method \cite{PSS}. 
\end{proof}

\begin{remark}
    Since $M$ is non-compact, one can choose $f$ with no index zero critical points. 
    \begin{enumerate}
        \item If $c_1(TM)=0$, then all our Floer complexes are supported between degree one and $\dim M$. Therefore, the product structure, which is degree compatible, is nilpotent.
        \item If $\omega$ is exact, the quantum product agrees with the usual product, which is nilpotent.
    \end{enumerate}
\end{remark}

In the definition of $SH^*_M(M)$, we use functions that have small negative slopes at infinity. If we use similar functions with small positive slope at infinity, we will get the usual cohomology $H^*(M)$. More precisely, consider Morse functions $g_n$ on $M$ such that 
\begin{enumerate}
    \item $g_n\leq g_{n+1}$ for any $n$.
    \item Each $g_n$ is a $C^2$-small Morse function on $\bar{M}$ and $\lim_n g_n(x)=0, \forall x\in \bar{M}$.
    \item Each $g_n$ is linear outside $\bar{M}$ with a small positive slope.
\end{enumerate}
Connect these functions with monotone homotopies to get a Floer one-ray $\mathcal{G}$. We equip $H(\widehat{tel^*}(\mathcal{G}))$ with a product structure which is isomorphic to $H^*(M)$ as a ring with respect to the quantum product. We choose $f_n<g_n$ for every $n$ and obtain a restriction map 
$$
r_M: H(\widehat{tel^*}(\mathcal{C}_f))\to H^*(\widehat{tel^*}(\mathcal{G}))
$$
which matches the map $H^*_c(M) \to H^*(M)$. By the functoriality of restriction maps we get the following.

\begin{proposition}\label{p:factor}
    For any compact subset $K$ of $M$, there exists a ring map $r_K: H^*(M;\Lambda)\to SH^*_M(K)$ such that $r_K\circ r_M$ equals the restriction map $SH^*_M(M)\to SH^*_M(K)$.
\end{proposition}
\begin{proof}
    Since $K$ is compact, we can assume it is contained in $\bar{M}$. Then we define $r_K$ as the restriction map from $H(\widehat{tel^*}(\mathcal{G}))$ to $SH^*_M(K)$. The rest follows from the functoriality of restriction maps. In other words, the following diagram commutes.
    $$
    \begin{tikzcd}
    SH^*_M(M)\cong H^*_c(M) \arrow[rr, bend left=20, "{r_M}"] \arrow[r]
    & H^*(M) \arrow[r, "{r_K}"]
    & SH^*_M(K)
    \end{tikzcd}
    $$
\end{proof}

\begin{proof}
    [Proof of Theorem \ref{thm-rfh-sh-vanish}]
    The set $K_M$ is non-compact, but our choice of Hamiltonian functions to define $SH^*_M(K_M)$ all have one-periodic orbits in a compact set. Moreover, $\bar{M}$ and $K_M$ are two domains in $M$ with a common boundary $\partial\bar{M}$. Therefore the Mayer-Vietoris sequence follows from the same proofs in \cite{Var21}. The only difference is that \cite{Var21} uses Novikov filtration but here we use action filtration on the Floer complexes. The algebraic results needed to deal with this difference (certain complex is acyclic) will be discussed in Proposition \ref{p:MV}, where we give a full proof of the Mayer-Vietoris property. The algebraic properties of the sequence are proved in the above propositions. 
\end{proof}

\subsection{Proof of Theorem \ref{thm-prequantum}}\label{sec-quantum-proof}

Now we move to the second extension of symplectic cohomology with support. Fix a closed symplectic manifold $(D,\omega_D)$ with an integral lift $\sigma$ of its symplectic class $[\omega_D]$. Assume that $[\omega_D]=\kappa c_1(TD)$ for some $\kappa>0$ and assume that the Chern number of any sphere $S^2\to D$ with positive symplectic area is at least $2$. Choose a principle circle bundle $p: C\to D$ with Euler class $-\sigma$ and connection one-form $\alpha$ on $C$ such that $d\alpha=p^*\omega_D$. This makes $(C,\xi:=\ker\alpha)$ into a contact manifold. Let $S(C)$ be the symplectization of $C$, which is canonically identified with $(C\times (0,\infty), d(r\alpha))$. For any compact subset $K$ of $S(C)$, we will define symplectic cohomology with support on $K$. Since $S(C)$ has a concave end, extra care need to be taken to avoid Floer cylinders escaping to the negative infinity. The main technique here is the compactness results used in \cite[Section 9.5]{CO}. Note that both the symplectic class and $c_1(TS(C))$ vanish here, so we simply have $\Lambda=\Bbbk$ and Hamiltonian Floer complexes are generated by contractible $1$-periodic orbits without having to specify the caps.

In the symplectization $S(C)$, a smooth function  $h: S^1\times S(C)\to \mathbb{R}$ is called \textit{$b$-admissible} for some $b\in\mathbb{R}$ if for every $t\in S^1$, 
\begin{itemize}
    \item $h_t:=h(t,\cdot): S(C)\to \mathbb{R}$ is constant and equal to some $h_t(-\infty)\in \mathbb{R}$ near the concave end, and
    \item $h_t$ is of the form $m_tr+d_t$ for $m_t,d_t\in \mathbb{R}$ near the convex end.
    \item $1$-periodic orbits of $h$ with action less than $b$ are non-degenerate and of finite number.
    \item Inside $\mathcal{L}S(C):=C^\infty(S^1,S(C))$ with $C^0$-norm, the locus $\mathcal{P}_{<b}(h)$ of $1$-periodic orbits of action less than $b$ and  the locus $\mathcal{P}_{\geq b}(h)$ of $1$-periodic orbits of action at least equal to $b$ are isolated in the sense that their $\epsilon$-neighborhoods are disjoint for some $\epsilon>0.$
\end{itemize} 
Note that if $h$ is $b$-admissible and $b'<b$, then $h$ is $b'$-admissible as well.

A typical example of a $b$-admissible function $h$ satisfies the following simpler to state properties: \begin{itemize}
    \item For some $r_-'<r_-\in \mathbb{R}$, in the region $r<r_-$, $h_t(r,x)=\rho(r)$ for some function $\rho:(-\infty, r_-)\to \mathbb{R}$ that is equal to $h(-\infty)>b$ for $r\leq r_-'$ and is concave on $(r_-',r_-)$.
    \item For some $r_+\in \mathbb{R}$, in the region $r>r_+$, $h_t(r,x)=mr+d$ for some $m,d\in \mathbb{R}$ near the convex end such that $m$ is not a Reeb period of $(C,\alpha)$.
    \item All $1-$periodic orbits of $h$ in the region $r_-\leq r \leq r_+$ are non-degenerate.
\end{itemize} 
Let us call such $b$-admissible functions \emph{typical} $b$-admissible functions. A typical $b$-admissible function $h$ has many degenerate orbits. For example, the constant orbits on $\{r=r'_-\}$ are not even of Morse-Bott type. We will exclude them from our Floer complexes by an action cutoff at level $b$.

We also assume that all almost complex structures used on $S(C)$ below are equal to a fixed cylindrical almost complex structure $J(-\infty)$ near concave end and to some cylindrical almost complex structure near the convex end.

\begin{lemma}\label{lem-gromov-compactness}
    Let $b$ be a real number, $h,h'$ two $b$-admissible functions with $h<h'$ and  $J,J'$ two $1$-periodically time dependent compatible almost complex structures on $S(C)$. Consider domain dependent monotone Floer data $(H_{k},J_k)$, $k=1,2,\ldots$ equaling $(h,J)$ and $(h',J')$ near the ends, which converge to Floer data $(H_{\infty},J_\infty)$ as $k\to \infty.$ Let $u_k: \mathbb{R}\times S^1\to M$, for $k=1, 2,\ldots$ satisfy the Floer equation for $(H_{k},J_k)$ with the same asymptotics $\gamma_-,\gamma_+$ whose actions are less than $b$. Assume that the images of all $u_k$ lie in a region $r>r_0.$ Then, $u_k$ has a subsequence converging to a broken cylinder $u_\infty$ consisting of solutions to the Floer equation of $(h,J)$, then a solution to that of $(H_{\infty},J_\infty)$ and then of $(h',J')$ (uniformly on all compact subsets, up to reparametrization as usual) going from $\gamma_-$ to $\gamma_+$. In particular, all the $1$-periodic orbits that appear as an asymptotic condition in the broken solution automatically have action less than $b.$ 
\end{lemma}
\begin{proof}
    Let us first make an important note. Let $u$ be a finite energy Floer solution for $(H,J)$ equaling $(h,J)$ and $(h',J')$ near the ends. Then, for any $\epsilon>0$, for $s$ sufficiently small $u_s:=u(s,\cdot)$ must be in an $\epsilon$ neighborhood of the locus of all $1$-periodic orbits of $h$ inside $\mathcal{L}S(C)$ and similarly for the output with primes. This follows immediately from \cite[Lemma 6.5.13]{AuD}.

    We now run the usual Gromov-Floer compactness argument (e.g. \cite[Lemma 9.1.7]{AuD}). Note that by the maximum principle and our assumption on the existence of $r_0$, in fact the images of all $u_k$'s lie in a compact subset. First, by compactness as in \cite[Theorem 6.5.4]{AuD} we extract a subsequence (and reparameterizations) of $u_k$ which converges to a finite energy Floer cylinder $v$ with output asymptotic $\gamma_+$. Now $v_s$ for sufficiently small $s$ is in the $\epsilon>0$ neighborhood of $1$-periodic orbits of $h$ inside $\mathcal{L}S(C)$, where we choose $\epsilon$ so small that the $\epsilon$-neighborhoods of each point in $\mathcal{P}_{<b}(h)$ and the $\epsilon$-neighborhood of $\mathcal{P}_{\geq b}(h)$ are all pairwise disjoint from each other. Hence, for all sufficiently small $s$, $v_s$ lies in exactly one of these neighborhoods. 

    We claim that it cannot lie inside that of $\mathcal{P}_{\geq b}(h)$. If it does, then by \cite[Proposition 6.5.7]{AuD} (also noting the Step 1 of the proof) we obtain a contradiction to the fact that our maps do not decrease action. If it lies in one of the other neighborhoods, then we conclude that $v$ has an input asymptote to a non-degenerate orbit with action less than $b.$

    We continue building the broken solution in the usual way. At each step, we prove that the input must asymptote to a non-degenerate orbit with action less than $b$ using the same argument. Eventually, we will construct a Floer cylinder that satisfies the Floer equation of $(H_{\infty},J_\infty)$ as the action keeps decreasing and we have only finitely many orbits with action less than $b$. Continuing from there to build our solution, we must end up at a solution that has $\gamma_-$ as an input asymptotic by the same reason. This finishes the proof.
\end{proof}

In the above lemma, a $C^0$-bound of Floer solutions is assumed near the concave end. Now we justify this assumption with the help of the minimal Chern number assumption on $D$. Our proof is essentially the same as \cite[Section 9.5]{CO} and employs a neck-stretching operation. The only difference is that our Boothby-Wang contact form is Morse-Bott non-degenerate, rather than non-degenerate. We give a detailed discussion for reader's convenience.

\begin{proposition}\label{p:nonescape}Let $b$ be a positive real number.
\begin{enumerate}
    \item Let $h$ be a $b$-admissible function on $S(C)$. For suitably chosen almost complex structures (omitted but present in the claims that follow), the action filtered $\Bbbk$-module $CF^*_{(-\infty, b)}(h)$ generated by the $1$-periodic orbits of $h$ with action less than $b$ can be equipped with a Floer differential that satisfies $d^2=0$. 
    \item Moreover, for another $b$-admissible function $h'>h$ and a suitable monotone homotopy connecting them, the continuation map $CF^*_{(-\infty, b)}(h)\to CF^*_{(-\infty, b)}(h')$ is a well-defined filtration non-decreasing chain map. 
    \item  Finally,  for $b$-admissible functions $h''>h'>h$ and monotone homotopies connecting $h$ and $h'$, $h'$ and $h''$ and also  $h$ and $h''$, we can construct a filtration non-decreasing chain homotopy between $CF^*_{(-\infty, b)}(h)\to CF^*_{(-\infty, b)}(h'')$ and the composition $CF^*_{(-\infty, b)}(h)\to CF^*_{(-\infty, b)}(h')\to CF^*_{(-\infty, b)}(h'')$.  
\end{enumerate}

\end{proposition}
\begin{proof}
     Let $h$ be a $b$-admissible function on $S(C)$, so there exists $r_0\in (0,\infty)$ such that $h_t(x,r)\equiv h_t(-\infty)$, for any $t\in S^1$, $r<r_0$ and $x\in C$. We assumed that any one-periodic orbit of $h$ with action less than $b$ is non-degenerate. We define the graded, filtered (by action) vector space $CF^*_{(-\infty, b)}(h)$ as stated in the proposition. Note that every point in the region ${r\leq r_0}$ is a degenerate constant orbit of $h$. As a result these constant orbits have action $\int_{0}^{1}h_t(-\infty)dt\geq b$ and are not taken as generators in $CF^*_{(-\infty, b)}(h)$. 
     %Moreover, they do not contribute to any Floer differential and continuation maps, by above Lemma \ref{lem-gromov-compactness}.

Next we will equip $CF^*_{(-\infty, b)}(h)$ with a Floer differential. Once we find a suitable almost complex structure, the matrix entries of the differential is defined by counting rigid (up to reparametrization), or equivalently, index $1$ solutions to the Floer equation as is usual in Floer theory. For some $r_0'<r_0$, fix the hypersurface $C_{r_0'}=\{r=r_0'\}$. We can choose a time-dependent almost complex structure $J$ that is time-independent and equal to $J(-\infty)$ in the region $\{r<r_0\}$ and perform the neck-stretching operation at $C_{r_0'}$, which gives us a sequence of almost complex structures $\{J_k\}_{k=1}^\infty$. Here $J(-\infty)$ should be of a particular type to do neck-stretching. We refer to \cite[Lemma 2.4]{CO} or \cite[Section 2.5]{CM2} for a detailed account of the operation for a similar purpose. We will make sure that $J$ is chosen sufficiently generic as needed in the following argument and then prove that for sufficiently large $k$ the differential is well-defined and squares to $0$. We do not spell out the genericity requirement on $J$ here as a stylistic choice. The actual Floer trajectories that we examine later will have asymptotics to non-degenerate Hamiltonian orbits. Recall that $J$ is fixed on $\{r<r_0\}$ to perform the neck-stretching and is allowed to vary, in a time-dependent way, near those Hamiltonian orbits. The formal proofs of transversality can be based on the standard argument in Floer theory, e.g \cite[Section 8]{SZ} or \cite[Section 4.2]{Schwarz}. In a similar situation, the transversality is proved in \cite[Corollary 7.6]{CM2}. 

We will first prove that for $k$ sufficiently large, we have only finitely many (up to reparametrization) Floer solutions between any two generators $\gamma_-,\gamma_+$ of $CF^*_{(-\infty, b)}(h)$ with $i(\gamma_+)=i(\gamma_-)+1.$ Note that, because of their actions, all of these generators are in the region $\{r>r_0\}$. Assume that the claim is not true, i.e. that there exists a sequence of $k$'s for which the compactness of the zero dimensional moduli spaces fails. For every $k$ in this sequence, we have Floer solutions that intersect $C_{r_0'}$ (otherwise Lemma \ref{lem-gromov-compactness} would imply that the zero dimensional moduli spaces are actually compact).
%This implies that for every $k$, we have Floer solutions that intersect $C_{r_0'}$ by Lemma \ref{lem-gromov-compactness}, which is a slightly improved version of Gromov compactness. 
Then, by SFT compactness \cite[Theorem 10.3]{BEHWZ} or \cite[Theorem 2.9]{CM2} (and finiteness of the number of generators), we can find a sequence of $J_k$-Floer solutions $u_k$ of index $1$ between some $\gamma_-,\gamma_+$ that has an SFT limit $u_\infty$, which is an SFT building consisting of components in three types of regions: the positive completion of $\{r<r_0'\}$, several levels of the symplectization of $\{r=r_0'\}$, and the negative completion of $\{r>r_0'\}$.  The components in $\{r<r_0'\}$ and in $\{r=r_0'\}$ are holomorphic buildings, which solve the Cauchy-Riemann equation with a time-independent cylindrical almost complex structure. The components in $\{r>r_0'\}$ are punctured broken Floer trajectories, which solve the Floer equation with respect to $h$ and a time-dependent complex structure we call $J_\infty$, with punctures asymptotic to Reeb orbits. We aim to prove that the components in $\{r<r_0'\}$ and $\{r=r_0'\}$ are empty, which gives a contradiction and proves the claim. To do this, we analyze the components in $\{r>r_0'\}$.

The building $u_\infty$ has components $u_\pm$ that lie in $\{r>r_0'\}$ and have asymptotes to $\gamma_\pm.$ We claim first that $u_+=u_-.$ If not, then we can find separating simple loops $\delta_k$ in the domain of $u_k$ (oriented as the boundary of the domain $D_k$ that contains the output) for sufficiently large $k$, whose projections to $C$ converge to a Reeb orbit $\tau$ (for some oriented parametrizations) and whose projections to $(0,\infty)$ lie in $r< r_0.$ The topological energy of $u_k$ restricted to $D_k$ is always positive and it is also supposed to equal $\mathcal{A}(\gamma_+)+\mathcal{A}(\delta_k)$ by Stokes Theorem, where $\mathcal{A}(\delta_k)$ is the integral  of $h_t(-\infty) dt+u_k^*(r\alpha)\in \Omega^1(\mathbb{R}\times S^1_t)$ on $\delta_k.$ Note that the integral of $f(t)dt$ (a closed $1$-form on the cylinder) along $\delta_k$ is $-\int_{0}^{1}f(t')dt'$ for any $f:S^1\to \mathbb{R}$ and the integral of $u_k^*(r\alpha)$ along $\delta_k$ converges to minus the period of $\tau$ times $r_0'$. Using $\int_{0}^{1}h_t(-\infty)dt\geq b$, this shows that $\mathcal{A}(\gamma_+)+\mathcal{A}(\delta_k)<0$ and gives a contradiction. We have shown that $u_+=u_-.$ Moreover, the punctures of $u_\pm$ that are asymptotic to Reeb orbits correspond to non-separating curves in the domain cylinder. This implies that each of them is capped with a (possibly broken) disk whose floors living in spaces that are all homeomorphic to the symplectization $S(C)$. Therefore, those asymptotic Reeb orbits are contractible in $S(C)$.

Now, consider the component $u_+$ of $u_\infty$ in $\{r>r_0'\}$, whose asymptotics consist of $\gamma_-,\gamma_+$ and $p\geq 0$ Reeb orbits that are elements of the Morse-Bott components $\gamma_i$, $1\leq i\leq p$. Recall the Reeb orbits of our contact form are multiple covers of circle fibers. The space of Reeb orbits is a disjoint union of manifolds. Here the Morse-Bott components mean the components of the disjoint union. The above argument shows that any Reeb orbit contained in $\gamma_i$ is contractible in $S(C)$, since the orbits in the same component are homotopic. Note that $u_+$ is an element of the moduli space $$\mathcal{M}(\gamma_-,\gamma_+; \gamma_i, 1\leq i\leq p; h, J_\infty)$$ of \emph{punctured Floer trajectories} up to the translation action of $\mathbb{R}$ in the domain cylinder (see page 31 of \cite{BO} for a very similar definition). The Fredholm index of this moduli space is
\begin{equation}\label{eq:index1}
\begin{split}
    & i(\gamma_+)-i(\gamma_-)-1-\sum_{i=1}^p(\text{CZ}(\gamma_i)-\dfrac{1}{2}\dim D+ \dfrac{1}{2}\dim S(C) -3)\\
    =& i(\gamma_+)-i(\gamma_-)-1-\sum_{i=1}^p(\text{CZ}(\gamma_i)-2).
\end{split}
\end{equation}
Here the term $\frac{1}{2}\dim D$ comes from the Morse-Bott nature of the Boothby-Wang contact form, and $\text{CZ}(\gamma_i)$ means the Robbin-Salamon index. This index computation is a combination of \cite[Equation (4)]{CM2} and \cite[Page 2103]{CO}. By choosing $J$ generically, we can make sure that these moduli spaces are regular for all choices of $\gamma_-,\gamma_+; \gamma_i, 1\leq i\leq p$. Note that $\gamma_\pm$ are non-degenerate Hamiltonian orbits and and we can use domain dependent almost complex structures near them.

If the minimal Chern number of $D$ is at least $2$, then $\text{CZ}(\gamma_i)$ is at least $4$ for any contractible $\gamma_i$, by \cite[Equation (2.11)]{BKK}. Hence, in particular, the moduli space to which $u_+$ belongs has a negative expected dimension if $p>0$. By regularity, $p$ must be zero. This implies that the part of $u_\infty$ in $\{r<r_0'\}$ and $\{r=r_0'\}$ is indeed empty and gives the desired contradiction.

This shows that for sufficiently large $k$, say $k\geq k_0,$ the differential is well-defined. We now show that for sufficiently large $k\geq k_0$, the differential squares to $0$. If this is not true, then it is due to the non-compactness of the moduli space of possibly broken index $2$ Floer cylinders for each $k\geq k_0.$ Hence, we can extract a sequence $u_k$ as above, with the only difference that they are now of index $2.$ We reach the same conclusion by considering Equation $(\ref{eq:index1})$ for $i(\gamma_+)-i(\gamma_-)=2$. This finishes the proof of (1).

To prove (2) we run the same argument for continuation maps. The only difference is that now the moduli spaces are defined without modding out the action of $\mathbb{R}.$ The Hamiltonian data and almost complex structures that go into the Floer equations become fully domain dependent (not just $t$ dependent). The non-negativity of the topological energy is still true due to the monotonicity requirement. Hence, the action arguments go through the same way. The dimension formula for the punctured Floer solutions now becomes  \begin{equation}
    i(\gamma_+)-i(\gamma_-)-\sum_{i=1}^p(\text{CZ}(\gamma_i)-2).
\end{equation} But, this time, for defining continuation maps, the degree difference is $0$, and for proving that we have a chain map, the degree difference is $1$. The result is that for $J$ and $J'$ that we constructed in the previous stage for $h$ and $h'$, and any monotone homotopy that does not admit an $s$-independent Floer solution (which exists in plenty), we can choose a generic (in a similar strong meaning as before) interpolating path $\mathcal{J}$ and do neck stretching at a sufficiently small $r_0'$.\footnote{The regularity of moduli spaces could be also achieved for $h'\geq h$ and moreover the locus where the two functions are equal is equal to the closure of an open subset; compare with \cite[Proposition 4.2.17]{Schwarz}.} We prove that for sufficiently large $k$, the compactness requirements are also satisfied.

The final statement in (3) about the chain homotopy requires working with parameter space $P=[0,1]$ and considering moduli spaces that live over $P$. The Floer continuation map data over $0$ and $1$ are given. We choose an interpolation of the Hamiltonian data $H_P$ such that for each $p\in (0,1)$, we have a cylinder dependent Hamiltonian which satisfies the monotonicity requirement (and does not admit $s$-independent solutions) and we choose a generic interpolation $J_P$ of the almost complex structures. Both are supposed to be given by gluing near the ends of $P.$ 

The matrix entries of the chain homotopy are defined by counting rigid solutions of the parametrized moduli problem. These rigid solutions go from a generator $\gamma_-$ of $CF^*_{(-\infty, b)}(h)$ to a generator $\gamma_+$ of $CF^*_{(-\infty, b)}(h'')$ with $i(\gamma_+)=i(\gamma_-)-1.$ Looking at the $1$-dimensional compactified solution space of this parametrized moduli problem, we prove that this is indeed a chain homotopy. These moduli spaces go between generators of the same degree. In order to get compactness, we again need to do neck stretching at a sufficiently small $r_0'.$ 

The part of the argument that relies on action is the same. For the second part, the key is to make sure that the moduli spaces $$\mathcal{M}(\gamma_-,\gamma_+; \gamma_i, 1\leq i\leq p; H_P, J_{P\infty})$$ that live over $P$ are regular as the solutions of a parametrized moduli problem. That is, the moduli spaces containing of pairs $(u,p)$ with $p\in P$ and $u$ in the Floer moduli space of $p$, are transversally cut out. The dimensions of these moduli spaces are now 
\begin{equation}
    i(\gamma_+)-i(\gamma_-)+1-\sum_{i=1}^p(\text{CZ}(\gamma_i)-2),
\end{equation}
where the $+1$ term reflects that $P=[0,1]$ is one-dimensional. Again, since whenever $i(\gamma_+)-i(\gamma_-)\leq 0$ and $p\geq 1,$ these moduli spaces are empty, the argument goes through.

\end{proof}

\begin{remark}\label{rem-parametrized-nonescape}In fact, for maps defined using continuation map data (possibly broken) parametrized by higher dimensional manifolds with corners $P$ (see \cite[Section 4.c.]{Se2} for a friendly introduction in the slightly simpler situation where $P$ is a manifold with boundary), the same argument continues to hold. The dimension formula becomes \begin{equation}
    i(\gamma_+)-i(\gamma_-)+\dim(P)-\sum_{i=1}^p(\text{CZ}(\gamma_i)-2).
\end{equation} The degree difference that is relevant for defining the map changes to $-\dim(P)$ and for the proof that it satisfies the desired relation to $-\dim(P)+1$. In other words, although our parameter space is of possibly high dimension, we only consider the index zero or index one solutions as parametrized moduli problems. Then the minimal Chern number assumption on $D$ tells us such solutions cannot admit punctures.
\end{remark}

Of course, there are other ways to achieve the above $C^0$-bound results. We give brief comments in the following remarks about how they are applicable to our setup.

\begin{remark}
    For a contact manifold with a non-degenerate contact form, there is an \textit{index-positivity} condition in \cite{CFO,CO} (specifically, Condition
(i) from \cite[Section 9.5]{CO}). By applying the same neck-stretching operation, \cite[Section 9.5]{CO} establishes the well-definedness of Hamiltonian Floer theory on the symplectization of a contact manifold with this property. Some further developments in this direction include \cite{Ue,BKK}. In our case, one can slightly perturb the Boothby-Wang form and try to apply their results. Under the minimal Chern number assumption, the carefully perturbed contact form (up to some arbitrarily high period bound) is index-positive, as shown in \cite[Section 5.3]{BKK}. Alternatively, to the Morse-Bott moduli spaces we used in our proof, we could either perform neck stretching along a slight perturbation of $C_{r_0'}$ or prove that doing Floer theory for arbitrarily small perturbations of $\alpha$ suffices for our purposes.
\end{remark}

\begin{remark}
    Let us note that admissible functions are more general than the usual V-shaped functions used to define Rabinowitz Floer homology \cite{CFO, Ue, CO, BKK}. The SFT compactness argument that we referred to in the proof relies on the elimination of two types of buildings. The first type are those where there are holomorphic planes in lower levels converging to Reeb orbits at convex ends, see \cite[Figure 2]{Ue}. These are eliminated by the index positivity condition in the same way throughout the literature (including here). The second type of buildings to be eliminated are those where the output is disconnected from the input in the building and the component containing the output (living at the top level) has an end converging to a Reeb orbit at the concave end, see \cite[Figure 3]{Ue}. In \cite{CFO, Ue}, using an argument that originates from \cite[Section 5.2]{BO}, these are eliminated by combining \cite[Lemma 2.3]{CO} and the maximum principle without having to do an action truncation. This argument does not apply to our case. On the other hand, the argument used in \cite[Proposition 9.17]{CO} does employ an action truncation and relies on a basic action computation. This one does apply in the generality that we require.
\end{remark}

On the other hand, the symplectization of a prequantization has a non-exact filling given by the disk bundle of the corresponding complex line bundle. An alternative proof to the compactness result is to prove that the Floer solutions considered above do not touch the zero section for suitably chosen almost complex structures. It is likely that such an approach is related to the wrapping-number-zero Rabinowitz Floer cohomology in \cite{AK}.

\begin{remark} Instead of working in the symplectization $S(C)$ and use neck stretching, we could also do a symplectic cut at $r=1$ and work inside the component that contains the convex end. This is equivalent to the negative quantum line bundle $E$ over $D$ with Euler class $-\sigma$ equipped with a Hermitian connection corresponding to $\alpha.$ Then, we would restrict ourselves to almost complex structures that are equal to a given one that makes the zero section $Z$ an almost complex submanifold in some neighborhood of $Z$. We consider Hamiltonian functions on $E$ that are constant in a neighborhood of $Z.$ We filter out the $1$-periodic orbits that occur along $Z$ artificially (and some others) by doing an action cutoff below a level that is lower than the value of the Hamiltonians involved at $Z$. We can try to do Floer theory with the remaining orbits that satisfy the action condition and considering only the Floer solutions that are disjoint from $Z$. Usual Gromov compactness holds in $E$ for Floer solutions with bubbles attached, but we claim that, under our assumptions, if we choose a sufficiently generic $J$, then for the orbits that we are interested in, the solutions counted in defining the differential, continuation maps, homotopies etc. are actually not bubbled and are entirely disjoint from $Z$. Moreover, the desired algebraic relations (e.g. $d^2=0$) are satisfied. The meaning of sufficiently generic is now given by an analogue of Cieliebak-Mohnke's \cite[Propositions 6.9 and 6.10]{CM} about Floer solutions with tangency conditions on the zero section. Take a Floer cylinder $u$ that appears as a component of a limit of Floer cylinders $u_k$ (that are rigid or belong to a $1$ dimensional family) disjoint from $Z$. Say that $u$ passes through $Z$ a positive number $p$ times with any fixed orders of tangency (equal to intersection multiplicity minus one). In the limiting configuration, at every intersection point, there must be an attached bubble in $Z$ and the total symplectic area of the attached bubble configuration must equal the intersection multiplicity. To see this, first note that the intersection number with $Z$ of a smooth map $v:(\Sigma, \partial\Sigma)\to (E,E\setminus Z)$, where $\Sigma$ is a compact oriented surface, is given by $\int_{\partial\Sigma}(\partial v)^*(r\alpha)-\int_{\Sigma}v^*\omega.$ Now, consider a closed disk in the cylinder that contains a single point $z$ that maps to $Z$ under $u$. If there is no bubbling, $u_k\mid_D$ uniformly converges to $u\mid_D$. This is a contradiction since the intersection number of $u\mid_D$ with $Z$ would have had to be zero, but it is positive by positivity of intersection. Moreover, by basic properties of Gromov convergence, $\int_{\partial D}(\partial u_k)^*(r\alpha)-\int_{D}u_k^*\omega$ (all equal to $0$) converges to $\int_{\partial D}(\partial u)^*(r\alpha)-\int_{D}u^*\omega$ minus the total symplectic area of the bubble components. This shows that the intersection number equals the total symplectic area of the bubble (configuration) as desired. For the $i$th intersection, we call the intersection number $I_i$ and $c_i$ denotes the Chern number of the attached bubble in $D$. The Chern number of the bubble inside $E$ becomes $c_i-I_i.$ The expected dimension of the relevant Cieliebak-Mohnke type moduli space becomes $0$ or $1$ minus $$\sum_{i=1}^p2(c_i-I_i)+2(I_i-1)=\sum_{i=1}^p2(c_i-1).$$ The action considerations allows us to also eliminate Floer solutions that are broken along an orbit contained in $Z$ essentially by the same argument. Hence, we reach to the same conclusions this way too.
\end{remark}

\begin{remark}
    Let us warn the reader about some elementary points regarding Proposition \ref{p:nonescape}. Let $h$ be $b$-admissible, $h'$ be $b'$-admissible and $h<h'$. We can define the continuation maps for different truncation levels $b, b'$: $CF^*_{(-\infty, b)}(h)\to CF^*_{(-\infty, b')}(h')$ as linear maps using the usual procedure. On the other hand, unless $b\geq b'$, the map might fail to be a chain map. It is possible for Floer trajectories between two allowed orbits to converge to a broken Floer trajectory consisting of a Floer solution for the differential of $CF^*_{(-\infty, b)}(h)$ with output not asymptotic to a non-degenerate orbit (for example, hypothetically, it could have an output asymptote to one of the constant orbits near the concave end) and a continuation map solution. Also note that the map for $b>b'$ can be factored into the continuation map from the Proposition \ref{p:nonescape} $CF^*_{(-\infty, b)}(h)\to CF^*_{(-\infty, b)}(h')$ and action truncation map  $CF^*_{(-\infty, b)}(h')\to CF^*_{(-\infty, b')}(h')$. 

    For chain homotopies as in Proposition \ref{p:nonescape} part (3), where in addition we also take a $h''$ that is $b''$-admissible and try to consider the case $b<b'<b''$, there could be an even more basic problem. It is possible for a Floer trajectory between two allowed orbits to be in the same one parameter family with a broken Floer trajectory where the middle orbit is filtered by the action truncation.
\end{remark}

For any compact subset $K$ and any $b\in\RR_+$, we choose the following $b$-acceleration data:
\begin{itemize}
    \item A sequence of $b$-admissible functions $h_n$ that monotonically approximate from below $\chi_{S^1\times K}^\infty$ with slopes at convex end $m_{n,t}\to +\infty$ for all $t\in S^1$.
    \item Monotone homotopies of functions connecting $h_n$'s and almost complex structures, which fit in Proposition \ref{p:nonescape}.
\end{itemize}
We have in our minds using typical $b$-admissible functions in practice. It is helpful to keep our Figure \ref{fig:filled cob} on page \pageref{fig:filled cob} in mind. Particularly, when $K=\{1\}\times C$ these functions can be chosen as the ``V-shaped Hamiltonians'' in \cite[Section 3.1]{BKK}.

By Proposition \ref{p:nonescape}, the Floer complexes $CF^*_{(-\infty, b)}(h_n)$ and continuation maps between them are well-defined. We use them to define a Floer telescope
$$
tel^*(h_{(-\infty, b)}):= tel(CF^*_{(-\infty, b)}(h_1)\to CF^*_{(-\infty, b)}(h_2)\to\cdots ).
$$ 
Next we want to take an inverse limit as $b$ goes to positive infinity. It will be more convenient to use an inverse telescope model for the \textit{homotopy inverse limit} as in \cite[Appendix A.4]{BSV}.

Consider an inverse system of chain complexes over $\Bbbk$:
$$
\mathcal{C}: C_1^*\xleftarrow{i_{12}} C_2^* \xleftarrow{i_{23}}\cdots.
$$
We define $\prod_l C_l^*$ as the degree-wise direct product of $C_l^*$'s. There is a natural chain map 
$$
id-i: \prod_l C_l^*\to \prod_l C_l^*, \quad (c_l)\mapsto (c_l-i_{l,l+1}(c_{l+1})).
$$
The inverse telescope complex is defined as 
$$
tel^*_{\leftarrow}(\mathcal{C}):= Cone(id-i)[-1].
$$
It always enjoys a Milnor exact sequence.

\begin{lemma}
[Lemma A.7 in \cite{BSV}]
    There is a short exact sequence
    $$
    0\to \varprojlim_l\nolimits^1 H^{j-1}(C_l^*)\to H^j(tel^*_\leftarrow(\mathcal{C}))\to \varprojlim_l H^j(C_l^*)\to 0.
    $$
\end{lemma}

Back to our Floer theoretic setup, we say that $\{h_n\}$ is an acceleration data for $K$ if there exists a sequence of positive real numbers $b_1<b_2<b_3<\ldots$ converging to $\infty$ such that for all $N=1,2,\ldots$, the subsequence $h_N, h_{N+1},\cdots$ is a $b_N$-admissible acceleration data for $K$. It is easy to construct such acceleration data using typical admissible functions; in particular, the constant value of $h_n$ at the concave end is larger than $b_n$. Hence we can construct the telescope $tel^*(h_{(-\infty, b_N)})$ using $h_N, h_{N+1},\cdots$ (forgive our slight abuse of notation). For all $N=1,2,\ldots$, there is a natural map

\[
\begin{tikzcd}
    tel^*(h_{(-\infty, b_{N+1})})=\oplus_{n=N+1}^\infty\left( CF^*_{(-\infty, b_{N+1})}(h_{n})\oplus CF^*_{(-\infty, b_{N+1})}(h_{n})[1]\right)\arrow{d}\\
    tel^*(h_{(-\infty, b_N)})=\oplus_{n=N}^\infty\left(CF^*_{(-\infty, b_N)}(h_{n})\oplus CF^*_{(-\infty, b_N)}(h_{n})[1]\right)
\end{tikzcd}
\]
which is the composition of the action truncation projection on Floer complexes of $h_{N+1}, h_{N+2},\cdots$  followed by the natural inclusion. These telescopes form an inverse system. We define the symplectic cohomology with support on $K$ with respect to the data chosen in the procedure as
\begin{equation}
    SH^*_{S(C)}(K,h):=H(tel^*_{\leftarrow}tel^*(h_{(-\infty,b_N)})).
\end{equation}

\begin{proposition}
    For different choices of acceleration datum $h$ and $h',$ we have preferred isomorphisms $SH^*_{S(C)}(K,h)\cong SH^*_{S(C)}(K,h')$ that are closed under compositions.
\end{proposition} 
\begin{proof}
    The same sandwiching argument \cite[Proposition 3.3.3]{Var21} applies here. The only difference is that we have to be careful about compactness near the concave end and, in particular, our inverse system is formed by telescopes that are not just truncations of each other. This inverse system does not satisfy the Mittag-Leffler property. Therefore, it is difficult to prove that the inverse limit of quasi-isomorphisms is a quasi-isomorphism. This is why we used the homotopy inverse limit in our definition, which is also as functorial as the usual inverse limit (more functorial in fact, but we don't need this for the purposes of this paper). As it immediately follows from the generalized Milnor exact sequence, the homotopy inverse limit of quasi-isomorphisms is automatically a quasi-isomorphism. We omit more details.
\end{proof}

We can also define restriction maps and prove that these restriction maps are closed under compositions with the preferred isomorphisms of the previous proposition using the well-known techniques used in \cite{Var21}. Hence, from now on we write $SH^*_{S(C)}(K)$ for $SH^*_{S(C)}(K,h)$ with any choice of $h.$

\begin{remark}
    We actually do not need the full strength of this independence of choices statement for the purposes of this paper, but it makes the proof cleaner, hence we include it.
\end{remark}

Next we show it enjoys the same properties of the symplectic cohomology with support where the ambient space is closed. We assume the reader is familiar to the original proofs in \cite{Var,Var21} and only indicate the necessary modifications.

\begin{proposition}
    If $K$ is displaceable in $S(C)$, then $SH^*_{S(C)}(K)=0$.
\end{proposition}
\begin{proof}
    We use the same twisting argument as in \cite[Section 4.2]{Var}. Since $K$ is compact, it can be displaced by a compactly supported Hamiltonian $\phi: S(C)\times [0,1]\to \mathbb{R}$. First consider a $b$-acceleration data $\{h_n\}$ consisting of typical $b$-admissible functions to construct $tel^*(h_{(-\infty,b)})$. We can assume the region where the $h_n$'s are constant at concave end or linear at convex end is far away from the support of $\phi$. Consider the twisting construction in \cite[Section 4.2.2]{Var}. For two Hamiltonian functions $h_n$ and $\phi$ we get a new Hamiltonian $h_n\bullet\phi:S(C)\times S^1\to \mathbb{R}$ such that $h_n$ is supported in $S(C)\times (0,1/2)$ and $\phi$ is supported in $S(C)\times (1/2,1)$. We can easily choose $h_n$'s so that the twisted functions $\{h_n\bullet \phi\}$ are $b$-admissible as well. We still have a well-defined Floer theory for $CF^*_{(-\infty, b)}(h_n\bullet\phi)$ and the Floer one-ray
    $$
    CF^*_{(-\infty, b)}(h_1\bullet\phi)\to CF^*_{(-\infty, b)}(h_2\bullet\phi)\to\cdots.
    $$
    By the topological energy estimate \cite[Proposition 4.2.7]{Var}, one can choose suitable homotopies in the above Floer one-ray such that the each continuation map raises the action at least $0.1$. Therefore the resulting telescope $tel^*((h\bullet\phi)_{(-\infty, b)})$ is acyclic.

    Without losing generality\footnote{Since $K$ is compact, we chose the displacing of $K$ via the isotopy of $\phi$ to be contained in a bigger compact region $W$. If $\min \phi<0$, then let $f:S(C)\to[0,+\infty)$ be a compactly supported non-negative Hamiltonian such that $f$ is constant and equals to $-\min \phi$ on an open neighbourhood of $W$. Denote by $f\#\phi$ the Hamiltonian that generates the composition of the isotopies of $f$ and $\phi$. Then, $f\# \phi$ displaces $K$, it is compactly supported, and $\min f\#\phi=0$.}, we assume that $\min\phi=0$. As in \cite[Proposition 4.2.3]{Var}, there exists a number $a\geq 1$ such that $h_n\bullet 0\leq h_n\bullet \phi\leq h_n\bullet 0+a\max\phi$ for all $n$. Consider the functions
    $$
    \{h_n\bullet 0\}_{n=1,2\cdots} \quad \text{and} \quad \{h_n\bullet 0+ a\max\phi\}_{n=1,2\cdots}, 
    $$where $0$ denotes the function that is identically $0.$
   Since the Floer theory of these functions are canonically identified with those of $h_n$ and $h_n+a\max\phi$ \cite[Lemma 4.2.1]{Var}, the telescopes fit into a sandwich 
    \begin{equation}\label{eq:sandwich}
        tel^*(h_{(-\infty, b)})\to tel^*((h\bullet\phi)_{(-\infty, b)})\to tel^*((h+a\max\phi)_{(-\infty, b)}).
    \end{equation}
    
    The middle telescope is acyclic as shown above. Hence the composition of these two maps induces a zero map on the homology level. On the other hand, there is a linear homotopy between $h_n$ and $h_n+a\max\phi$. This linear homotopy induces an identity map at the chain level, which only translates the action. It gives the projection map on homology
    \begin{equation}\label{eq:projection}
        H(tel^*(h_{(-\infty, b)}))\to H(tel^*(h_{(-\infty, b-a\max\phi)})),
    \end{equation}
    since the chain complexes $tel^*((h+a\max\phi)_{(-\infty, b)})$ and $tel^*(h_{(-\infty, b-a\max\phi)})$ are canonically identified.
    By the contractibility of Floer data, the composition of the two maps in $(\ref{eq:sandwich})$ induces the same map as in $(\ref{eq:projection})$. Therefore we proved that $(\ref{eq:projection})$ is the zero map. In particular, the argument shows that the projection map
    $$
    H(tel^*(h_{(-\infty, b)}))\to H(tel^*(h_{(-\infty, b-c)}))
    $$
    is zero for any $b$ and any constant $c$ which is larger than $a\max\phi$.

    Finally, in the definition of $SH^*_{S(C)}(K)$, we can chose a sequence $b_l$ such that $b_{l+1}-b_l>a\max\phi$ for any $l$. Then we get an inverse system with all maps $H(tel^*(h_{(-\infty, b_{l+1})}))\to H(tel^*(h_{(-\infty, b_l)}))$ being zero. The Milnor exact sequence tells us $SH^*_{S(C)}(K)=0$.
    
\end{proof}

\begin{proposition}\label{p:MV}
    Let $K_1,K_2$ be two compact subsets of $S(C)$ which are Poisson commuting. There is a Mayer-Vietoris exact sequence
    $$
    \cdots\to SH^k_{S(C)}(K_1\cup K_2)\to SH^k_{S(C)}(K_1)\oplus SH^k_{S(C)}(K_2)\to SH^k_{S(C)}(K_1\cap K_2)\to \cdots.
    $$
\end{proposition}
\begin{proof}
For a fixed $b$, consider $b$-acceleration data $\{H^A_{s,t}\}$ for any $A\in \{K_1,K_2,K_1\cup K_2,K_1\cap K_2\}$, such that $H^{A_1}_{n,t}\geq H^{A_2}_{n,t}$ whenever $A_1\subset A_2$. These data give a Floer three-ray (keep in mind Remark \ref{rem-parametrized-nonescape})
$$
\mathcal{F}_{<b}:= F_{1,<b}\to F_{2,<b}\to\cdots
$$ 
with
$$
\mathcal{C}^A_{<b}:= CF^*_{(-\infty,b)}(H^A_{1,t})\to CF^*_{(-\infty,b)}(H^A_{2,t})\to \cdots
$$
on the four infinite edges. Each slice $F_{n,<b}$ of this three-ray is a square which looks like
$$
F_{n,<b}=
\begin{tikzcd}
    CF^*_{(-\infty,b)}(H^{K_1\cup K_2}_{n,t}) \arrow[r, "r^1_n"] \arrow[d, "r^2_n"] \arrow[rd, "g_n"]
    & CF^*_{(-\infty,b)}(H^{K_1}_{n,t}) \arrow[d, "f^1_n"] \\
    CF^*_{(-\infty,b)}(H^{K_2}_{n,t}) \arrow[r, "f^2_n"]
    & CF^*_{(-\infty,b)}(H^{K_1\cap K_2}_{n,t})
\end{tikzcd}.
$$
Here all the arrows are restriction maps induced by the chosen homotopies. See \cite[Section 3.4]{Var21}. By using these maps, we can form the double cone $Cone^2$ of this square. The underlying complex of the double cone is
$$
\begin{aligned}
    & D_{n,<b}:= Cone^2(F_{n,<b})\\
    := & CF^*_{(-\infty,b)}(H^{K_1\cup K_2}_{n,t})\oplus CF^*_{(-\infty,b)}(H^{K_1}_{n,t})[1]\oplus CF^*_{(-\infty,b)}(H^{K_2}_{n,t})[1]\oplus CF^*_{(-\infty,b)}(H^{K_1\cap K_2}_{n,t}).
\end{aligned}
$$ 
Its differential is defined as
$$
\partial_n=
    \begin{bmatrix}
    d^{K_1\cup K_2}_n & 0 & 0 & 0\\
    -r^1_n & -d^{K_1}_n & 0 & 0\\
    r^2_n & 0 & -d^{K_2}_n & 0\\
    g_n & f^1_n & f^2_n & d^{K_1\cap K_2}_n\\
    \end{bmatrix}.
$$
Therefore we have a new Floer one-ray
$$
\mathcal{D}_{<b}:= D_{1,<b}\to D_{2,<b}\to \cdots.
$$

Next we show $Cone^2\circ tel^*(\mathcal{F}_{<b})$ is acyclic. Note that $Cone^2\circ tel= tel\circ Cone^2$. Hence it suffices to show $tel^*(\mathcal{D}_{<b})$ is acyclic. Since $tel^*(\mathcal{D}_{<b})$ is quasi-isomorphic to $\varinjlim_{n}D_{n,<b}$ and direct limit preserves exactness, we only need to show each $D_{n,<b}$ is acyclic.

The differential $\partial_n$ of $D_{n,<b}$ is a sum of Floer differentials and continuation maps, which both respect the symplectic action, since our homotopies are monotone. So there is a decomposition $\partial_n=\partial_n^0 +\partial_n^+$ where $\partial_n^0$ consists of the contributions from Floer trajectories with zero topological energy. By action considerations, $(\partial_n)^2=0$ implies $(\partial_n^0)^2=0$. We claim that there exist acceleration data $\{H^A_{s,t}\}$ such that  $(D_{n,<b}, \partial_n^0)$ is an acyclic complex. Under the Poisson commuting condition, such acceleration data have been explicitly constructed in \cite[Sections 5.5-5.9]{Var} when the ambient space is compact. However, the nontrivial part, construction of the compatible boundary accelerators, only happens locally near the boundaries of approximating domains of $K_1, K_2$. Hence we can repeat the construction inside a bounded domain, making sure that the functions are all equal near the boundary of this domain. It is easy to extend the functions so that they are admissible and equal outside of this domain. We can then do the relevant perturbations as in the closed case.

Next, given those acceleration data, there exists $\epsilon_n>0$ such that every entry in $\partial_n^+$ raises a positive action greater than $\epsilon_n$, because each $D_{n,<b}$ is a finite-dimensional module over the ground ring. Then for any $a\in\RR$, we consider $D_{n,[a,b)}$ as the subcomplex of $D_{n,<b}$ containing generators with action larger than or equal to $a$. Define a filtration on $D_{n,[a,b)}$ by
$$
F^l D_{n,[a,b)}:= D_{n,[a+l\epsilon_n,b)}, \quad l=0,1,\cdots.
$$
This is a bounded filtration on $D_{n,[a,b)}$, hence the induced spectral sequence converges to $H(D_{n,[a,b)})$. The first page of this spectral sequence is computed by using maps with zero topological energy. As mentioned above, just using those maps gives vanishing homology groups. So the first page of the spectral sequence is zero, which shows that the final page $H(D_{n,[a,b)})=0$. To see this, we consider the zero map from $D_{n,[a,b)}$ to a trivial complex, which contains a single element $\{0\}$ with $+\infty$ action in each degree. The zero map is a filtered map between two complete and exhaustive complexes. Then we use the Eilenberg-Moore Comparison Theorem 5.5.11 in \cite{We}. Since $D_{n,<b}$ is the direct limit of $D_{n, [a,b)}$, it is also acyclic.

By \cite[Lemma 2.5.3]{Var}, if
$$
Cone^2(tel^*_\leftarrow tel^*(\mathcal{F}_{<b}))
$$
is acyclic, then we get the desired Mayer-Vietoris exact sequence
$$
\cdots\to H(tel^*_\leftarrow tel^*(\mathcal{C}_{<b}^{K_1\cup K_2}))\to H(tel^*_\leftarrow tel^*(\mathcal{C}_{<b}^{K_1}))\oplus H(tel^*_\leftarrow tel^*(\mathcal{C}_{<b}^{K_2}))\to H(tel^*_\leftarrow tel^*(\mathcal{C}_{<b}^{K_1\cap K_2}))\to\cdots.
$$
Note that $Cone^2$ commutes with $tel^*_\leftarrow$, it is equivalent to show $tel^*_\leftarrow(Cone^2\circ tel^*(\mathcal{F}_{<b}))$ is acyclic. We already have that $Cone^2\circ tel^*(\mathcal{F}_{<b})$ is acyclic. The rest follows from the generalized Milnor exact sequence.

\end{proof}

\begin{comment}
    \begin{proposition}
    For any compact subset $K$ of $SC$, there is a ring structure on $SH_{SC}(K,b)$ for any $b\in\RR_+$. For $K'\subset K$, the restriction map $SH_{SC}(K,b)\to SH_{SC}(K',b)$ respects the unit.
\end{proposition}
\begin{proof}
    Algebraic structures on symplectic cohomology with support have been constructed in \cite{TVar, AGV}. Here since we truncate the action from above, the construction follows from more classical method. For example, see \cite[Section 10]{CO}.
\end{proof}

\end{comment}

When $K=\{1\}\times C$, our symplectic cohomology with support is isomorphic to the Rabinowitz Floer homology considered by Bae-Kang-Kim \cite{BKK}.

\begin{proposition}
    For any degree $k$, $SH^k_{S(C)}(K=\{1\}\times C)$ is isomorphic to the Rabinowitz Floer homology $SH_*(C)$ in \cite[Corollary 1.8]{BKK} at a certain degree $*$. This isomorphism holds over any coefficient.
\end{proposition}
\begin{proof}
    As we mentioned above, in this case our Hamiltonian functions are the same as theirs, see Remark \ref{r:Morse-Bott} below. For a fixed Conley-Zehnder index, all the Reeb orbits of the Boothby-Wang contact form have a bounded period. As we explain in what follows, this enables us to compute the Rabinowitz Floer homology as
    $$
    SH_*(C)= H(\varprojlim_b\varinjlim_a\varinjlim_n CF_{*,[a,b)}(h_n)).
    $$ 
    Recall that the Hamiltonian orbits we use to define $SH_*(C)$ are all at the bottom level of the V-shaped Hamiltonians. Hence bounded Reeb periods imply that their symplectic actions are uniformly bounded. See \cite[Equation (4.27)]{BKK} and its explanation. In particular, when $b$ is large and $a$ is small enough, the first two limits just stabilize for a fixed degree. Therefore we get 
    $$
    SH_*(C)= H(\varprojlim_b\varinjlim_a\varinjlim_n CF_{*,[a,b)}(h_n))= H(\varinjlim_n CF_{*,[a_0,b_0)}(h_n))\cong H(tel^k(h_{(-\infty, b_0)}))
    $$
    for large $b_0$ and small $a_0$, since telescope is quasi-isomorphic to the direct limit. The degree shift between $*$ and $k$ is due to different grading conventions.

    By the same reason, the projection map $H(tel^k(h_{(-\infty, b')}))\to H(tel^k(h_{(-\infty, b)}))$ stabilizes when $b'>b>b_0$. The Milnor exact sequence tells us $SH^k_{S(C)}(K)\cong H(tel^k(h_{(-\infty, b_0)}))$. Here, we used that $ \varprojlim_b{}\!\!^1 H(tel^k(h_{(-\infty, b)}))$ vanishes due to the Mittag-Leffler condition, which is satisfied because the projection map stabilizes.
\end{proof}

\begin{proof}
    [Proof of Theorem \ref{thm-prequantum}]
    If all fibers are contact displaceable, then they are displaceable in $S(C)$, viewed as subsets of $\{1\}\times C$. Hence we can cook up a finite Poisson commuting cover of $\{1\}\times C$ in $S(C)$, whose elements are compact and displaceable. By an iterated use of the Mayer-Vietoris property we get that $SH^*_{S(C)}(\{1\}\times C)=0$. This implies the Rabinowitz Floer homology $SH_*(C)=0$. Hence, the Gysin long exact sequence from \cite[Corollary 1.8]{BKK} implies that $I_\sigma$ is an isomorphism. This contradicts the assumptions of the theorem and finishes the proof.
\end{proof}

\begin{remark}\label{r:Morse-Bott}
    Strictly speaking, Bae-Kang-Kim \cite{BKK} uses a Morse-Bott model for Hamiltonian Floer complexes, while the usual setup \cite{Var21} for symplectic cohomology with support uses non-degenerate Hamiltonians. Now these models can be easily translated between each other. For example, see \cite[Appendix A]{RZ} and references therein.
\end{remark}

\subsection{Proof of Theorem \ref{thm-elem-cont}}

In this section, we freely use the notation introduced in the statement of Theorem \ref{thm-elem-cont}.
View $G:=p^{-1}(L)$ as a principal circle bundle over $L$ with the flat connection $\beta:=\iota_G^*\alpha$. Identify $C$ with $\{1\}\times C$ in $S(C)$. Suppose that $\{1\}\times G$ is not Hamiltonian displaceable from itself in $S(C)$, but $Z$ is contact displaceable from $G$. Let $\phi$ be the time-$1$ map of a contact isotopy displacing $Z$ from $G$. It lifts to a Hamiltonian diffeomorphism $\Phi: S(C)\to S(C)$. Let $N$ be an open neighborhood of $Z$ in $C$ such that $\phi(N)$ is disjoint from $G$. Therefore, $\Phi((0,\infty)\times N)$ does not intersect $(0,\infty)\times G$. By Lemma \ref{l:conformal}, there exists a number $A$ such that $\Phi(\{r\}\times C)\subset (0,e^A r)\times C$ for every $r>0$. 

Let us denote the principal circle bundle $p\mid_G: G\to L$ by $\pi.$

\begin{lemma}
    There is a smooth function $f: G\to \mathbb{R}$ such that
    \begin{itemize}
        \item $f$ is equal to its minimal value $-1$ on $G\setminus N$,
        \item $f$ is invariant under parallel transport of the connection $\beta$, and
        \item for all $b\in L$, the integral of $f\beta$ on $p^{-1}(b)$ is zero.
    \end{itemize}  
\end{lemma}
\begin{proof}
Let $b$ be an arbitrary point in $L$. Because the connection $\beta$ is flat, the parallel transport maps on $\pi:G\to L$ depend only on the homotopy classes of paths relative to their boundary. Because $Z$ is tangent to the horizontal subbundle of the connection and $Z\to L$ is a finite covering map, by the path lifting property for covering spaces (and the definition of parallel transport maps), we see that the image of the holonomy homomorphism $\pi_1(L,b)\to S^1$ has finite image. Let us denote the image by $hol_b$. 

Let $c$ be the circle obtained by taking the quotient of  $\pi^{-1}(b)$ under the action of $hol_b$. Let $U\subset c$ be an open arc such that when we take its preimage in $\pi^{-1}(b)$ and parallel transport to all fibers, it stays inside $N$. Such a $U$ exists due to the compactness of $L.$ Let $t$ be an angular coordinate on $c$ such that the pullback under the quotient map of $dt$ is $\beta\mid_{\pi^{-1}(b)}$. 

    We choose a smooth function on $c$ that is equal to its minimum value $-1$ outside of $U$ and whose integral is $0$ with respect to the measure $dt$. We then lift this function to $\pi^{-1}(b)$ by precomposing with the quotient map and parallel transport it to all the other fibers. This is a smooth function because our principal circle bundle with its connection is locally trivial and smoothness is a local property. This finishes the construction of the desired $f$. 
\end{proof}

A direct consequence is that $f\beta$ represents the zero class in $H^1(G;\RR)$, which now we explain. Since $f$ is invariant under parallel transport, the one-form $f\beta$ is closed. To see this note that $$d(f\beta)=df\wedge \beta+fd\beta.$$ 
At any point in $G$, the tangent space splits into a one-dimensional vertical space and a horizontal space. If we pair $d(f\beta)$ with two vectors, then the
first term vanishes because both $df$ and $\beta$ vanish on horizontal tangent vectors and the vertical tangent space is one-dimensional. The second term vanishes because $\beta$ is closed. 

In order to finish the proof of the claim, we need to show that for any loop $\gamma: S^1\to G$, we have $\int_\gamma f\beta=0$. Define $\tau:= \pi\circ \gamma$ and consider the pullback principal circle bundle $\tau^*\pi$ over $S^1$ with the pullback connection $\tilde{\beta}.$ Let $\tilde{f}$ be the pullback of $f$ to the total space of $\tau^*\pi$. By construction, $\tau^*\pi$ has section that gives $\gamma$ after postcomposition to the canonical map to $G.$ It suffices to show that the closed $1$-form $\tilde{f}\tilde{\beta}$ (because it is the pullback of a closed one form) is exact. By the integral $0$ property of $f$, we see that the integral over a fiber is $0$. By taking the preimage of $Z$, we also see that $\tau^*\pi$ admits a parallel multisection. A connected component of the multisection (oriented arbitrarily) has homology class that is not a multiple of the fiber class. The form $\tilde{f}\tilde{\beta}$ vanishes identically when restricted to this component, finishing the proof of exactness.

%Because the holonomy is contained in a finite cyclic group, any loop $\gamma$ in $G$ is homologous to a sum of multiple of the fiber loop and a loop tangent to $\ker\beta$, and therefore, we have $\int_\gamma f\beta=0$.

\begin{proof}
[Proof of Theorem \ref{thm-elem-cont}]

For $t\in [0, 1)$, consider the submanifolds 
$$
G_t:=\{(r,x)\mid x\in G, r=1+tf(x)\}\subset S(C).
$$
Since $G_0$ is a Lagrangian submanifold and $f\beta$ is closed, each $G_t$ is a Lagrangian submanifold (cf. \cite[Section 2.2]{EHS}). Indeed, denoting by $j_t:G\to S(C)$ the ``parametrization'' $j_t(x):=(1+tf(x), x)$ of $G_t$, the pull-back of the symplectic form $\omega=d(r\alpha)$ from $S(C)$ via $j_t$ is equal to
\[j_t^\ast\omega= dj_t^\ast(r\alpha) = d(r\circ j_t\cdot \beta) = d((1+tf(x))\cdot \beta) = d\beta +td(f\beta).\] 
Now, $d\beta=0$ because $\beta$ is a flat connection on $G$ (this, on the other hand, follows from $L$ being a Lagrangian) and $t d(f\beta)=0$ because $f\beta$ is closed. Hence, $j_t^\ast\omega=0$ and $G_t$ is a Lagrangian submanifold (here, we also used $2\dim G_t=\dim S(C)$). 
Moreover, we know that $f\beta$ is exact because it represents the zero class in $H^1(G;\RR)$. Since
\[ \omega(\partial_tj_t, dj_t)= (dr\wedge\alpha + rd\alpha)(f(x)\partial_r, dj_t(x)) = f\cdot j_t^\ast\alpha = f\beta \]
(and since $f\beta$ is exact), by \cite[Lemma 2.3]{AS}, we can find a Hamiltonian isotopy $\Psi_t$ that generates the family restricted to any closed interval $[0,t_0]$ with $t_0<1$.

For $1>t_0>1-e^{-A}$, for the $A$ that was chosen above, we have
$$
G_{t_0}\subset ((0,e^{-A})\times C)\cup ((0,\infty)\times N).
$$  
If $x\in G_{t_0}\cap ((0,\infty)\times N)$, then $\Phi(x)$ does not intersect $(0,\infty)\times G$ because $\phi$ displaces $N$ from $G$. If $x\in G_{t_0}\cap ((0,e^{-A})\times C)$ then $\Phi(x)$ does not intersect $\{1\}\times G$ by the  choice of $A$ (recall the last sentence of the first paragraph of this section). Therefore $\Phi(\Psi_{t_0}(\{1\}\times G))$ is disjoint from $\{1\}\times G$, a contradiction.
\end{proof}

\begin{comment}
    \begin{remark}
    The argument above also shows that: for any compact subset $K\subset S(C)$, if $K$ is Hamiltonian non-displaceable from $\{1\}\times G$ then it is Hamiltonian non-displaceable from $(0,\infty)\times Z$. In fact, one can even show that if $K$ is Hamiltonian displaceable from $(0,\infty)\times Z$, then it is Hamiltonian displaceable from $(a,\infty)\times G$ for any $a>0.$
\end{remark}
\end{comment}

\appendix

\section{Comparison results}
\label{sec:comparison}

Symplectic cohomology with support unifies most of the existing symplectic cohomology theories by specializing the support set $K$. In this section, we prove several such comparison results in the most widely used case where the ambient symplectic manifold $M$ is a Liouville manifold. These results are known by the experts (essentially all going to back to \cite{Vit}), but we include them here for completeness as it can be hard to find definitive statements in the literature.

\subsection{Liouville subdomains}\label{sec:comparison-1}
We start with Liouville subdomains in Liouville manifolds. The following proposition is immediate from definitions.

\begin{proposition}\label{p:rel2vit}
    For a Liouville manifold $M$ that is the symplectic completion of a Liouville domain $\bar{M}$, the symplectic cohomology with support $SH^*_M(\bar{M})$ is isomorphic to the Viterbo symplectic cohomology $SH^*(M)$ over any coefficient.
\end{proposition}

Next we want to prove the following locality result. 
\begin{proposition}\label{p:locality}
    For any Liouville subdomain $\bar{N}$ of a Liouville manifold $M$, there is an isomorphism, called the locality isomorphism,
    $$
    I_{NM}: SH^*_M(\bar{N})\to SH^*(N),
    $$
    where $N$ is the Liouville completion of $N$ and $SH^*(N)$ is the symplectic cohomology defined by Viterbo \cite{Vit,Se}. This isomorphism holds over any coefficient ring $\Bbbk$.
\end{proposition}

Suppose that $M$ is the symplectic completion of a Liouville domain $\bar{M}$ and let $\lambda$ be the Liouville form on $M$. Without losing generality, we assume that $\bar{N}\subset \bar{M}$ and $\alpha:=\lambda\mid_{\partial\bar{N}}$ is a non-degenerate contact form. Since $\bar{N}$ is a Liouville subdomain, the restriction of $\lambda$ to $\bar{N}$ makes it into a Liouville domain and there is a symplectic embedding of its symplectic completion $N$ into $M$. By the non-degeneracy condition, let the action spectrum of $\alpha$ be
$$
Spec(\alpha)= \{s_1, s_2, \cdots\}, \quad 0<s_1<s_2<\cdots, \quad \lim_{i}s_i= +\infty.
$$
Pick a sequence of numbers $c_i$ such that $s_{i}< c_i<s_{i+1}$ for all $i\geq 1$. We will construct Hamiltonian functions that compute $SH^*_M(\bar{N})$.

\begin{figure}
  \begin{tikzpicture}[yscale=0.8]
  \draw [->] (0,0)--(6,0);
  \draw (1,0)--(2,1)--(5,1);

  \draw (1,0)--(3,4)--(5,4);

  \draw [dotted] (2,1)--(2,0);
  \node [below] at (2,0) {$R_{n}$};
  \draw [dotted] (3,0)--(3,4);
  \node [below] at (3,0) {$R_{n+1}$};
  \node [below] at (1,0) {$1$};
  \node [right] at (6,0) {$r$};
  \node [right] at (5,1) {$f_{n}$};

  \node [right] at (5,4) {$f_{n+1}$};

  \end{tikzpicture}
  \caption{Hamiltonian functions in the cylindrical coordinate.}\label{fig:Ham}
\end{figure}
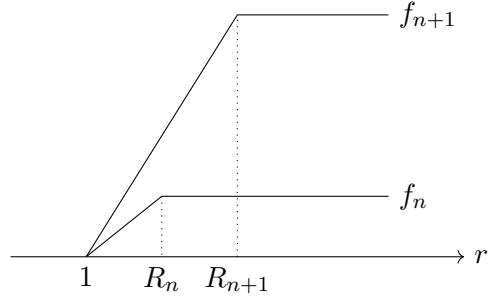

Fix a small number $\epsilon>0$. For any number $R_1>1$, consider a smooth function $f_1: [0,+\infty)\to \RR$ such that
\begin{enumerate}
    \item $f_1(r)$ is a negative constant when $r<1-\epsilon$.
    \item $f_1'(r)\geq 0$ for all $r$.
    \item $f_1(r)= c_1r$ when $1+\epsilon<r<R_1-\epsilon$.
    \item $f_1(r)$ is convex when $1-\epsilon<r<1+\epsilon$.
    \item $f_1(r)$ is concave when $R_1-\epsilon<r<R_1+\epsilon$.
    \item $f_1(r)$ is a constant when $r>R_1 +\epsilon$.
\end{enumerate}
Then view $f_1$ as a Hamiltonian function on $[1, +\infty)\times \partial\bar{N}$. The one-periodic orbits of $f_1$ fall into two groups: those in $\{r<1+\epsilon\}$ which are called lower orbits, and those in $\{r>R_1-\epsilon\}$ which are called upper orbits.

By Viterbo's $y$-intercept computation, the lower orbits of $f_1$ all have negative actions. The action of upper orbits is bounded below by
$$
-s_1R_1 +c_1(R_1-1)= R_1(c_1-s_1)- c_1
$$
up to $\epsilon$. Since $c_1> s_1$, we can choose $R_1$ large enough such that the above action is larger than one. This complete the construction of $f_1$. Next we construct $f_2$ in a similar way, such that 
\begin{enumerate}
    \item $f_2$ satisfies $(1)-(6)$ for $f_1$ for some $R_2>R_1$, replacing $c_1$ by $c_2$ in $(3)$.
    \item $f_2\geq f_1$.
    \item $R_2$ is chosen that all upper orbits of $f_2$ have action greater than two.
\end{enumerate}
Repeating this process we construct a sequence $\{ f_n\}$ such that
\begin{enumerate}
    \item $f_n$ satisfies $(1)-(6)$ for $f_1$ for some $R_n>R_{n-1}$, replacing $c_1$ by $c_{n}$ in $(3)$.
    \item $f_{n}\geq f_{n-1}$.
    \item $f_n$ converge to zero on $\bar{N}$ and diverge to positive infinity outside $\bar{N}$.
    \item $R_n$ is chosen that all upper orbits of $f_n$ have action greater than $n$.
\end{enumerate}
A rough depiction of $f_n$ is in Figure \ref{fig:Ham}. The functions $f_n$ are constant outside a compact subset of $N$, so we can extend it to $M$ constantly.

Finally we perturb $f_n$ into a non-degenerate Hamiltonian $F_{n,t}$ on $M$ such that
\begin{enumerate}
    \item Outside some $\bar{M}\cup_{\partial\bar{M}}([1,C]\times \partial\bar{M})$, $F_{n,t}$ only depends on the cylindrical coordinate, and it has small slope.
    \item The non-constant one-periodic orbits of $F_{n,t}$ come from the non-constant one-periodic orbits of $f_n$ by using a time-dependent perturbation to break the $S^1$-symmetry. $F_{n,t}$ does not have non-constant periodic orbits other than those ones.
    \item Constant orbits of $F_{n,t}$ are obtained by adding $C^2$-small Morse functions to $f_n$.
    \item $F_{n,t}\leq F_{n+1}$, and $F_{n,t}$ converge to zero on $\bar{N}$ and diverge to positive infinity outside $\bar{N}$.
\end{enumerate}
The perturbations above can be chosen as small as needed. Hence all one-periodic orbits of $F_{n,t}$ also fall into two groups: lower orbits which are in $\{r<1+\epsilon\}$ and upper orbits which are in $\{r>R_n+\epsilon\}$. Moreover, the lower orbits all have negative actions and the upper orbits of $F_{n,t}$ have actions larger than $n$.

\begin{proof}[Proof of Proposition \ref{p:locality}]
    Choosing monotone homotopies connecting $F_{n,t}$ and $F_{n+1,t}$, we get a Floer one-ray
    $$
    \mathcal{C}= CF^*(F_{1,t})\to CF^*(F_{2,t})\to\cdots.
    $$
    Let $CF_+(F_{n,t})$ be the free $\ZZ$-module generated by upper orbits of $F_{n,t}$. Since the Floer differential does not decrease action, it is a subcomplex of $CF(F_{n,t})$. We write $CF_-(F_{n,t}):= CF(F_{n,t})/CF_+(F_{n,t})$. Then there are two more Floer one-rays
    $$
    \mathcal{C}_+= CF^*_+(F_{1,t})\to CF^*_+(F_{2,t})\to\cdots, \quad \mathcal{C}_-= CF^*_-(F_{1,t})\to CF^*_-(F_{2,t})\to\cdots.
    $$
    One can check that they induce an exact sequence
    $$
    0\to \widehat{tel^*}(\mathcal{C}_+)\xrightarrow{i} \widehat{tel^*}(\mathcal{C})\xrightarrow{p} \widehat{tel^*}(\mathcal{C}_-)\to 0.
    $$
    The continuation maps in $\widehat{tel^*}(\mathcal{C}_+)$ raise the action of any fixed generator to positive infinity, since any upper orbit of $F_{n,t}$ has action at least $n$. Therefore $\widehat{tel^*}(\mathcal{C}_+)$ is acyclic and $p$ is an quasi-isomorphism, see \cite[Lemma 5.3]{DGPZ}.

    By linearly extending the lower part of $F_{n,t}$, we get a sequence of Hamiltonians $F_{n,t}'$ on $N$, which gives us a telescope $tel^*(\mathcal{C}')$. The classic symplectic cohomology of $N$ can be computed as $SH^*(N)= H(tel^*(\mathcal{C}'))$. On the other hand, since all lower orbits of $F_{n,t}$ have negative action, we have $\widehat{tel^*}(\mathcal{C}_-)= tel^*(\mathcal{C}_-)$.
    
    The underlying generators of $tel^*(\mathcal{C}_-)$ can be identified with the generators of $tel^*(\mathcal{C}')$. Moreover, we can choose the almost complex structures in the cylindrical region in $N$ to be also cylindrical. Then the maximum principle tells us the images of all Floer differentials and continuation maps are in $N$, see Lemma 2.2 in \cite{CO}. Hence the differentials and continuation maps in $tel^*(\mathcal{C}_-)$ are identified with those in $tel^*(\mathcal{C}')$. In conclusion, we have 
    $$
    SH^*_M(\bar{N})= H(\widehat{tel^*}(\mathcal{C}))\cong H(\widehat{tel^*}(\mathcal{C}_-))= H(tel^*(\mathcal{C}_-))= H(tel^*(\mathcal{C}'))=  SH^*(N).
    $$
    We write $I_{NM}$ as the map induced by the quasi-isomorphism $p$.
\end{proof}

Next we show the isomorphism $I_{NM}$ is compatible with Viterbo restriction maps.

\begin{figure}
  \begin{tikzpicture}[yscale=0.6]
  \draw [->] (0,0)--(9,0);
  \draw (0,-0.2)--(1,-0.2);
  \draw (1,-0.2)--(5,4);
  \draw (5,4)--(7,4)--(8,4.5);
  \draw (0,-0.5)--(4,-0.5)--(6,1.5)--(7,1.5)--(8,2);

  \node [right] at (8,4.5) {$G_{n,t}$};
  \node [right] at (8,2) {$F_{n,t}$};
  \draw [dotted] (1,-1)--(1,3);
  \draw [dotted] (4,-1)--(4,3);
  \draw [dotted] (7,-1)--(7,4);
  \node [below] at (1,-1) {$\partial\bar{W}$};
  \node [below] at (4,-1) {$\partial\bar{N}$};
  \node [below] at (7,-1) {$\partial\bar{M}_C$};
  \end{tikzpicture}
  \caption{Hamiltonian functions for restriction maps.}

  \label{fig:res}
\end{figure}
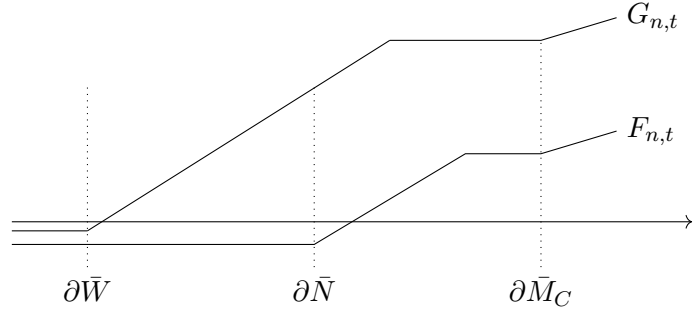

\begin{proposition}\label{p:restriction}
    Let $\bar{N}$ be a Liouville subdomain of $\bar{M}$, and let $\bar{W}$ be a Liouville subdomain of $\bar{N}$. We have the following commutative diagram
    $$
    \begin{tikzcd}
    SH^*_M(\bar{N}) \arrow[r, "I_{NM}"] \arrow[d, "r"]
    & SH^*(N) \arrow[d, "r_V"] \\
    SH^*_M(\bar{W}) \arrow[r, "I_{WM}"]
    & SH^*(W)
    \end{tikzcd}.
    $$
    The left vertical map is the restriction map between symplectic cohomology with support, the right vertical map is the restriction map defined by Viterbo \cite{Vit}.
\end{proposition}
\begin{proof}
    Let $\{F_{n,t}\}$ be the Hamiltonian functions used in the locality isomorphism for $\bar{N}$ in Proposition \ref{p:locality}. We will construct a family of Hamiltonian functions $\{G_{n,t}\}$, which will be used to study the relation between locality isomorphisms and restriction maps. See Figure \ref{fig:res} for a pictorial depiction.\footnote{There the graph of $G_{n,t}$ is a bit misleading. On $\bar{N}\setminus (\bar{N}\cap  W_{C'})$, it is extended by constants.}

    Since $\bar{W}$ is a Liouville subdomain of $\bar{N}$, there is a tower of symplectic embeddings $W\to N\to M$ of complete Liouville manifolds. We use the functions $f_n$'s in Proposition \ref{p:locality} to construct non-degenerate Hamiltonian functions $G_{n,t}$ on $M$ such that
    \begin{enumerate}
        \item $G_{n,t}\geq F_{n,t}$.
        \item $G_{n,t}$ is $C^2$-small in $\bar{W}$.
        \item $G_{n,t}$ only depends on the cylindrical coordinate of $W$ on $W\setminus\bar{W}$. 
        %{\color{red}\item $G_{n,t}$ only depends on the cylindrical coordinate of $N$ outside some $N_C:=\bar{N}\cup_{\partial\bar{N}}([1,C]\times \partial\bar{M})$.}
        \item $G_{n,t}$ only depends on the cylindrical coordinate of $M$ outside some $M_{C}:=\bar{M}\cup_{\partial\bar{M}}([1,C]\times \partial\bar{M})$, and $G_{n,t}$ has a small slope there.
    \end{enumerate}
    The construction of $G_{n,t}$ is similar to the construction of $F_{n,t}$. Let $W_{C'}:=\bar{W}\cup_{\partial\bar{W}}([1,C']\times \partial\bar{W})$. First, we use $f_n$ on the pair $(W_{C'},\bar{W})$ to get a function which is $C^2$-small in $\bar{W}$, cylindrical on $W_{C'}\setminus \bar{W}$, and has zero slope at $\partial W_{C'}$. The number $C'$ will be chosen large enough for later action considerations. Particularly, $W_{C'}$ may not be contained in $\bar{N}$ or $\bar{M}$. However, $W_{C'}$ is contained in $M_C$ for some large $C$. Next we extend the function from $W_{C'}$ to $M_C$ by constants, and then perturb. The Liouville form $\lambda$ is assumed to be a non-degenerate contact from on $\partial\bar{W}$. Hence we can choose the slope of $G_{n,t}$ properly such that there is a neck region near $\partial\bar{W}$ where $G_{n,t}$ does not have one-periodic orbits, similar to the construction of $F_{n,t}$. All one-periodic orbits of $G_{n,t}$ which are in $\partial\bar{W}$ are called lower orbits, and all others are called upper orbits. Then similar to the construction of $F_{n,t}$, we further assume that
    \begin{enumerate}
        \item All upper orbits of $G_{n,t}$ have action larger than $n$, and all lower orbits of $G_{n,t}$ have negative action.
        \item $G_{n,t}\leq G_{n+1,t}$ for all $n$.
        \item $G_{n,t}$ converges to zero on $\bar{W}$ and diverge to positive infinity outside $\bar{W}$.
    \end{enumerate}
    This uses that the whole $W$ is embedded in $M$, so the neck parameter $C'$ could be chosen large enough. Note that the upper orbits consist of non-constant orbits near $\partial W_{C'}$ and constant orbits outside $W_{C'}$. The former has large action by the long neck computation, and the later has large Hamiltonian values. Then choosing monotone homotopies between $G_{n,t}$ and $G_{n+1,t}$, we get a Floer one-ray 
    $$
    \mathcal{G}= CF^*(G_{1,t})\to CF^*(G_{2,t})\to\cdots
    $$
    and $H(\widehat{tel^*}(\mathcal{G}))= SH^*_{M}(\bar{W})$. 

    By the assumption of actions of upper orbits, there is a projection map $p_W$ which gives the locality isomorphism $I_{WM}$ in the homology level, see Proposition \ref{p:locality}. We have three more maps
    $$
    \begin{aligned}
        & p_N: \text{the projection map for $\bar{N}$ given by the funtions $F_{n,t}$},\\
        & r: \text{the chain level restriction map from $\widehat{tel^*}(\mathcal{C})$ to $\widehat{tel^*}(\mathcal{G})$},\\
        & r_V: \text{the chain level Viterbo restriction map from $\widehat{tel^*}(\mathcal{C}')$ to $\widehat{tel^*}(\mathcal{G}')$}.
    \end{aligned}
    $$
Here $\mathcal{C}'$ is obtained by linearly extending the lower parts of Hamiltonian functions in $\mathcal{C}$, and $\mathcal{G}'$ is obtained by linearly extending the lower parts of Hamiltonian functions in $\mathcal{G}$, as in the end of proof of Proposition \ref{p:locality}. The resulting two Floer one-rays $\mathcal{C}'$ and $\mathcal{G}'$ are on $N$ and $W$ respectively. The restriction map $r$ is defined by choosing monotone homotopies between $F_{n,t}$ and $G_{n+1,t}$ and counting Floer continuation maps. Pick an element $x\in \widehat{tel^*}(\mathcal{C})$, we decompose it into lower generators and upper generators $x=x_- +x_+$. Similarly, write $r(x_-)=a_- +a_+, r(x_+)=b_+$. Then $p_W(r(x))= a_-$. On the other hand, the Viterbo restriction map $r_V$ is defined as the composition
    $$
    \widehat{tel^*}(\mathcal{C}') \cong \widehat{tel^*}(\mathcal{C}_-) \xrightarrow{r} \widehat{tel^*}(\mathcal{G}) \xrightarrow{p_W} \widehat{tel^*}(\mathcal{G}_-) \cong \widehat{tel^*}(\mathcal{G}').
    $$
    More precisely, by couting Floer continuation maps, we have $r: \widehat{tel^*}(\mathcal{C}) \to \widehat{tel^*}(\mathcal{G})$, and the Viterbo restriction map is the induced map between two quotients, see Section 2 in \cite{Vit}. Hence we get $r_V(p_N(x))=r_V(x_-)=a_-=p_W(r(x_-))$.

\end{proof}

\subsection{Filled Liouville cobordisms}\label{sec:comparison-2}

Now we compare our invariant with the \textit{symplectic homology of filled Liouville cobordisms} by Cieliebak-Oancea \cite{CO}. The proof uses the same strategy as above. We first setup nice acceleration data then apply homological algebra, to commute homology with limits.

\begin{definition}
    A \textit{Liouville cobordism} $(\bar{W}, \bar{\eta})$ is a compact manifold with boundary $\bar{W}$, equipped with a one-form $\bar{\eta}$ which has the following two properties. First, $\bar{\omega}=d\bar{\eta}$ is a symplectic form on $\bar{W}$. Second the vector field $V$, the Liouville vector field determined by $\bar{\omega}(V, \cdot)= \bar{\eta}$ is transversal to $\partial \bar{W}$. The component of $\partial \bar{W}$ where $V$ points outward is called the positive (convex) end of $\bar{W}$, and the component of $\partial \bar{W}$ where $V$ points inward is called the negative (concave) end of $\bar{W}$. Particularly, a Liouville domain is a Liouville cobordism without negative ends.
\end{definition}

\begin{comment}
    Given a Liouville domain $(\bar{M},\bar{\lambda})$, the one-form $\alpha:= \bar{\lambda}\mid_{\partial\bar{M}}$ gives a contact structure $\xi=\ker\alpha$ on $\partial\bar{M}$. Conversely, a Liouville filling of a contact manifold $(C,\xi)$ is a Liouville domain $(\bar{M},\bar{\lambda})$ such that $(\partial\bar{M},\ker\alpha)=(C,\xi)$. The symplectic completion of $(\bar{M}, \bar{\lambda})$ is
$$
M:= \bar{M}\cup_{\partial\bar{M}} ([1,+\infty)\times \partial\bar{M})
$$
with $\lambda:= \bar{\lambda}$ on $\bar{M}$ and $\lambda:= r\alpha$ on $[1,+\infty)\times \partial\bar{M}$ where $r$ is the coordinate on the first factor. 

For two Liouville domains $(\bar{M}, \bar{\lambda})$ and $(\bar{N}, \bar{\sigma})$ of the same dimension, if there is an embedding $\iota: \bar{N}\to \bar{M}$ with $\iota^* \bar{\lambda}= \bar{\sigma}$ then the image of $\iota$ is a Liouville subdomain of $(\bar{M}, \bar{\lambda})$. Such an embedding also induces a (not necessarily proper) embedding $\iota: (N, \sigma)\to (M, \lambda)$ between completions.
\end{comment}

\begin{definition}
    A filled Liouville cobordism is a Liouville cobordism $(\bar{W}, \bar{\eta})$ together with a Liouville domain $(\bar{M}, \bar{\lambda})$, such that $\partial_+(\bar{M}, \bar{\lambda})= \partial_-(\bar{W}, \bar{\eta})$. Hence we can glue them to get a larger Liouville domain $\bar{M}\cup \bar{W}$.
\end{definition}

Let $(\bar{W}, \bar{\eta})$ be a Liouville cobordism with a Liouville filling $(\bar{M}, \bar{\lambda})$. We glue them together to get a larger Liouville domain $(\bar{Y}, \bar{\tau})$. Let $(Y, \tau)$ be the completion of $(\bar{Y}, \bar{\tau})$ and consider a family of admissible Hamiltonian functions $H_n$ on $Y$, such that
\begin{enumerate}
    \item $H_n\leq H_{n+1}$ for all $n$.
    \item $H_n$ converges to zero on $\bar{W}$ and diverges to infinity on $Y-\bar{W}$.
\end{enumerate}
See Figure \ref{fig:filled cob}.

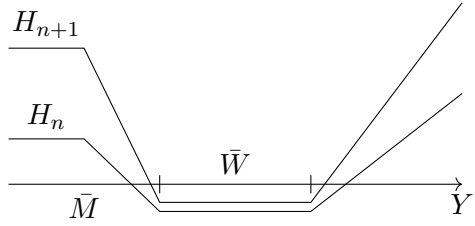
\begin{figure}
  \begin{tikzpicture}[yscale=0.6]

  \draw [->] (0,0)--(6,0);
  \draw (2,0.2)--(2,-0.2);
  \draw (4,0.2)--(4,-0.2);
  \draw (0,3)--(1,3)--(2,-0.4)--(4,-0.4)--(6,4);
  \draw (0,1)--(1,1)--(2,-0.6)--(4,-0.6)--(6,2);

  \node [above] at (0.5,3) {$H_{n+1}$};
  \node [above] at (0.5,1) {$H_n$};

  \node [above] at (3,0) {$\bar{W}$};
  \node [below] at (6,0) {$Y$};
  \node [below] at (1,0) {$\bar{M}$};

  \end{tikzpicture}
  \caption{Hamiltonian functions for a filled cobordism.}\label{fig:filled cob}
\end{figure}

\begin{definition}[Definition 2.8 in \cite{CO}]\label{d:cob}
    The symplectic homology of the Liouville cobordism $(\bar{W}, \bar{\eta})$ with a Liouville filling $(\bar{M}, \bar{\lambda})$ is defined as
    $$
    SH_*(W):= \varinjlim_{a\to -\infty}\varprojlim_{b\to +\infty}\varinjlim_{n\to \infty} H(CF_{*,[a,b)}(H_n)).
    $$
\end{definition}

Our grading and action functional conventions are the same as in \cite[Section 3.1]{Var}, which are different from those in \cite{CO}. The above definition has been translated into our action convention. What is important is the order of taking limits: one first takes direct limit of the slope of the Hamiltonian, then takes inverse limit on action upper bound and then takes direct limit on action lower bound.

\begin{proposition}\label{p:filled}
    Over any field $\KK$, the symplectic homology of the filled cobordism $SH_*(W;\KK)$ is isomorphic to $SH^*_Y(\bar{W};\KK)$ up to a grading shift.
\end{proposition}
\begin{proof}
    Pick an acceleration datum $\{H_{n,t}\}$ to compute $SH^*_Y(\bar{W};\KK)$. We have
    $$
    SH_*(W;\KK)= \varinjlim_{a\to -\infty}\varprojlim_{b\to +\infty}\varinjlim_{n\to +\infty}H(CF_{*,[a,b)}(H_{n,t})).
    $$

    Now we show that for a special choice of $\{H_{n,t}\}$, several limits in the definition commute. We construct $H_{n,t}$ in the following way.
    \begin{enumerate}
        \item In the interior of $\bar{W}$, the functions $\{H_{n,t}\}$ are $C^2$-small negative Morse functions.
        \item Near $\partial_-(\bar{W}, \bar{\eta})$, pick a neighborhood $[1-\epsilon,1]\times\partial_-\bar{W}$. Consider a decreasing function $f_n$ in the collar coordinate that is concave on $[1-\epsilon,1-2\epsilon/3]\times\partial_-\bar{W}$, linear on $[1-2\epsilon/3,1-\epsilon/3]\times\partial_-\bar{W}$ and convex on $[1-\epsilon/3,1]\times\partial_-\bar{W}$. The slope of the linear part is not in $\pm Spec(\bar{\eta}\mid_{\partial_-\bar{W}})$. $H_{n,t}$ is a small perturbation of $f_n$ to break the $S^1$-symmetry.
        \item On the cylindrical region $[1,+\infty)\times\partial_+\bar{W}$ we use an increasing convex function which is linear at infinity.
        \item In $\bar{M}-([1-\epsilon,1]\times\partial_-\bar{W})$, the functions $H_{n,t}$ are Morse functions with small derivatives.
    \end{enumerate}
    See Figure \ref{fig:filled cob}. The one-periodic orbits of $H_{n,t}$ can be divided into upper and lower orbits, determined by their Hamiltonian values. By Viterbo's $y$-intercept formula, we can assume that the actions of upper orbits of $H_{n,t}$ are larger than $n$. Since direct limit commutes with homology, we have
    $$\varinjlim_{n\to +\infty}H(CF_{*,[a,b)}(H_{n,t}))\cong H(\varinjlim_{n\to +\infty}CF_{*,[a,b)}(H_{n,t})).
    $$ 
    For a special choice of $\{H_{n,t}\}$, Lemma \ref{l:finite} below shows that 
    $$
    \varinjlim_{n\to +\infty}CF_{*,[a,b)}(H_{n,t}), \quad H(\varinjlim_{n\to +\infty}CF_{*,[a,b)}(H_{n,t}))
    $$ are finite-dimensional vector spaces over $\KK$ for any $[a,b]$. The finite-dimensional condition on the chain level implies the Milnor exact sequence \cite[Theorem 3.5.8]{We}
    $$
    0\to \varprojlim_b{}^1 H(\varinjlim_{n\to +\infty}CF_{*,[a,b)}(H_{n,t}))\to H(\varprojlim_b \varinjlim_{n\to +\infty}CF_{*,[a,b)}(H_{n,t}))\to \varprojlim_b H(\varinjlim_{n\to +\infty}CF_{*,[a,b)}(H_{n,t}))\to 0.
    $$
    Moreoever, the finite-dimensional condition on the homology level shows that the $\varprojlim_b\nolimits^1$-term is zero \cite[Exercise 3.5.2]{We}. Hence we get
    $$
    \varprojlim_{b\to +\infty}H(\varinjlim_{n\to +\infty}CF_{*,[a,b)}(H_{n,t}))\cong H(\varprojlim_{b\to +\infty}\varinjlim_{n\to +\infty}CF_{*,[a,b)}(H_{n,t})).
    $$
    which gives
    $$
    \varinjlim_{a\to -\infty}\varprojlim_{b\to +\infty}H(\varinjlim_{n\to +\infty}CF_{*,[a,b)}(H_{n,t}))\cong H(\varinjlim_{a\to -\infty}\varprojlim_{b\to +\infty}\varinjlim_{n\to +\infty}CF_{*,[a,b)}(H_{n,t})).
    $$
    Then we use Lemma \ref{l:quasi-comp} and Equation $(\ref{eq:defSH2})$ below to obtain that
    $$
    \varinjlim_{a\to -\infty}\varprojlim_{b\to +\infty}H(\varinjlim_{n\to +\infty}CF_{*,[a,b)}(H_{n,t}))\cong H(\varprojlim_{b\to +\infty}\varinjlim_{a\to -\infty}tel_*(\mathcal{C})_{[a,b)})
    $$
    where the later is Equation $(\ref{eq:defSH})$ form  page~\pageref{eq:defSH} up to a grading shift.
\end{proof}

\begin{remark}
An inverse system of finite-dimensional vector spaces always satisfy the Mittag-Leffler condition, which implies the vanishing of the $\varprojlim\nolimits^1$-term. However, an inverse system of finitely-generated free $\ZZ$-modules may not satisfy the Mittag-Leffler condition, see Remark 2.3 in \cite{ACF}.
\end{remark}

Now we prove the claims used above. The first one is geometric.

\begin{lemma}\label{l:finite}
    For particularly chosen $\{H_{n,t}\}$, we have that
    $$
    \varinjlim_{n\to +\infty}CF_{*,[a,b)}(H_{n,t}), \quad H(\varinjlim_{n\to +\infty}CF_{*,[a,b)}(H_{n,t}))
    $$ 
    are finite-dimensional vector spaces over $\KK$ for any fixed $a,b$.
\end{lemma}
\begin{proof}
    When constructing $\{H_{n,t}\}$, we can further assume the following. In the interior of $\bar{W}$, we have $H_{n,t}+ s_n=H_{n+1,t}$ for all $n\geq 1$ and some small positive number $s_n$. Hence the numbers of lower constant orbits of $H_{n,t}$ are the same for all $n$.
    
    On the other hand, the lower non-constant orbits of $H_{n,t}$ come from perturbations of Reeb orbits on $\partial_-\bar{W}, \partial_+\bar{W}$. Each Reeb orbit gives two Hamiltonian orbits after perturbation. Since we assume the contact forms on $\partial_-\bar{W}, \partial_+\bar{W}$ are non-degenerate, there are only finitely many Reeb orbits with action less than a given bound. Hence for any $a,b$, the numbers of lower one-periodic orbits of $H_{n,t}$ with actions in $[a,b)$ are uniformly bounded from above, independent of $n$. 
    
    For upper orbits of $H_{n,t}$, our special construction says that they all have actions larger than $n$. Therefore they escape from the action window $[a,b)$ when $n$ is large.
    
    In conclusion, for any $a,b$ the complex $CF_{*,[a,b)}(H_{n,t})$ is a finite-dimensional vector space over $\KK$. When we vary $n$, their dimensions are uniformly bounded from above. Hence the direct limit is finite-dimensional and so is its homology.
\end{proof}

The second claim is algebraic.  For any $a<b$, recall $tel^*(\mathcal{C})_{[a,b)}:= tel^*(\mathcal{C})_{\geq a}/tel^*(\mathcal{C})_{\geq b}$ by using the $\min$-action. On the other hand, we have a Floer one-ray
$$
\mathcal{C}_{[a,b)}:= CF^*_{[a,b)}(H_1)\to CF^*_{[a,b)}(H_2)\to \cdots.
$$
One can check that $tel^*(\mathcal{C})_{[a,b)}= tel^*(\mathcal{C}_{[a,b)})$. Therefore an equivalent definition of the completion is
\begin{equation}\label{eq:defSH2}
   \widehat{tel^*}(\mathcal{C})= \varprojlim_{b\to +\infty}\varinjlim_{a\to -\infty}tel^*(\mathcal{C}_{[a,b)}).
\end{equation}

\begin{lemma}\label{l:quasi-comp}
The two complexes
$$
\varprojlim_{b\to +\infty}\varinjlim_{a\to -\infty}tel^*(\mathcal{C}_{[a,b)})\quad \text{and}\quad \varprojlim_{b\to +\infty}\varinjlim_{a\to -\infty}\varinjlim_{n\to +\infty} CF^*_{[a,b)}(H_n)
$$
are quasi-isomorphic. The two complexes
$$
\varinjlim_{a\to -\infty}\varprojlim_{b\to +\infty}tel^*(\mathcal{C}_{[a,b)})\quad \text{and}\quad \varinjlim_{a\to -\infty}\varprojlim_{b\to +\infty}\varinjlim_{n\to +\infty} CF^*_{[a,b)}(H_n)
$$
are quasi-isomorphic. The two complexes 
$$
\varprojlim_{b\to +\infty}\varinjlim_{a\to -\infty}tel^*(\mathcal{C}_{[a,b)})\quad \text{and}\quad \varinjlim_{a\to -\infty}\varprojlim_{b\to +\infty}tel^*(\mathcal{C}_{[a,b)})
$$
are isomorphic.
\end{lemma}
\begin{proof}

The first two quasi-isomorphisms follow from Lemma 2.2.4 in \cite{Var}, where a Novikov filtration is used. However the proof is purely algebraic in nature, and hence it works for our action filtration. The third isomorphism follows from Theorem 3.2 and Lemma 4.2 in \cite{CF}, since the two limits are taken with respect to the same filtration.
\end{proof}

\begin{comment}
    Set 
    $$
    T_b:= \varinjlim_{a\to -\infty}tel^*(\mathcal{C})_{[a,b]}, \quad C_b:= \varinjlim_{a\to -\infty}\varinjlim_{n\to +\infty} CF^*_{[a,b]}(H_n).
    $$
    For any $b$, there is a map $i_b: T_b\to C_b$ which is a quasi-isomorphism and it is compatible with the inverse system with respect to $b$, see Lemma \ref{l:quasi}. Hence we have another inverse system $Cone(i_b)$ over $b$ such that the cone $Cone(i_b)$ is acyclic for any $b$. Then the proof in Lemma \ref{l:acyclic} tells that $\varprojlim_{b\to +\infty}Cone(i_b)$ is acyclic. The map $\varprojlim_{b\to +\infty}i_b$ gives the desired quasi-isomorphism.

    The same argument shows that $\varprojlim_{b}tel^*(\mathcal{C})_{[a,b]}$ and $\varprojlim_{b}\varinjlim_{n\to +\infty} CF^*_{[a,b]}(H_n)$ are quasi-isomorphic. Then the second quasi-isomorphism follows from direct limit $\varinjlim_{a}$ is exact.
\end{comment}

The direct limit $\varinjlim_{n} CF^*_{[a,b)}(H_n)$ can be verified to be finite-dimensional for fixed $a,b$ in certain cases, as we have shown above. However the telescope model is infinite-dimensional even when we fix an action window. This finishes the proof of Proposition \ref{p:filled}.

\subsubsection{Rabinowitz Floer homology}

Let $M$ be a Liouville manifold which is the symplectic completion of a Liouville domain $\bar{M}$. Then one can define the Rabinowitz Floer homology $RFH_*(\partial\bar{M})$ of $\partial\bar{M}$ in $M$. It is shown in \cite{CFO,CO} that $RFH_*(\partial\bar{M})$ is isomorphic to the symplectic homology $SH_*(\partial\bar{M})$ of the trivial cobordism $\partial\bar{M}\times[1,1+\epsilon]$ in $M$, filled by $\bar{M}$. Therefore Proposition \ref{p:filled} gives the following corollary. 

\begin{corollary}\label{c:RFH}
    Over any field $\KK$, we have an isomorphism
    $$
    RFH_*(\partial\bar{M};\KK)\cong SH^*_M(\partial\bar{M};\KK),
    $$
    up to a grading shift.
\end{corollary}

\bibliographystyle{amsplain}

\end{document}